\documentclass[12pt]{elsarticle}

\usepackage{mystyleElsArticle}
\usepackage{color}
\newcommand{\STATE}{\State}
\newcommand{\FOR}{\For}
\newcommand{\ENDFOR}{\EndFor}

\makeatletter
\def\ps@pprintTitle{%
  \let\@oddhead\@empty
  \let\@evenhead\@empty
  \let\@oddfoot\@empty
  \let\@evenfoot\@oddfoot
}
\makeatother
\usepackage[margin=0.94in]{geometry}    % Margin: 1-inch all round

\usepackage{scalerel,stackengine}
\stackMath
\newcommand\reallywidehat[1]{%
	\savestack{\tmpbox}{\stretchto{%
			\scaleto{%
				\scalerel*[\widthof{\ensuremath{#1}}]{\kern.1pt\mathchar"0362\kern.1pt}%
				{\rule{0ex}{\textheight}}%WIDTH-LIMITED CIRCUMFLEX
			}{\textheight}% 
		}{2.4ex}}%
	\stackon[-6.9pt]{#1}{\tmpbox}%
}

\begin{document}

\def \R {\mathbb{R}}
\def \a {\alpha}
\def \H {\mathcal{H}}
\def \b {\beta}
\def \N {\mathbb{N}}
\def \D {\mathcal{D}}

\begin{frontmatter}
  \title{Solving high-dimensional nonlinear filtering problems using a tensor train decomposition method\\
  \vspace{10pt}	
  \small Dedicated to Professor Thomas Kailath on the occasion of his 85th birthday}
   
%  \author[hku]{Dingjiong Ma}
%  \ead{martin35@hku.hk}
   \author[hku]{Sijing Li}
   \ead{lsj17@hku.hk}
   \author[hku]{Zhongjian Wang}
   \ead{ariswang@connect.hku.hk}
%   \author[beihang]{Xue Luo\corref{cor2}}
%   \ead{xluo@buaa.edu.cn}
   \author[Tsinghua]{Stephen S.T. Yau\corref{cor1}}
   \ead{yau@tsinghua.edu.cn}
   \author[hku]{Zhiwen Zhang\corref{cor1}}
   \ead{zhangzw@hku.hk}
%   \address[uci]{Department of Mathematics, University of California at Irvine, Irvine, CA 92697, USA.}
   \address[hku]{Department of Mathematics, The University of Hong Kong, Pokfulam Road, Hong Kong SAR, China.}
   %\address[beihang]{School of Mathematics and System Sciences, Beihang University, Beijing 100191, PR China.}
   \address[Tsinghua]{Department of Mathematics, Tsinghua University, Beijing 100084, China.}
   %\address[pnnl]{Computational Mathematics, Pacific Northwest National Laboratory, Richland, WA 99352, USA.}
 
%  \author[cms]{Thomas Y. Hou\corref{cor1}}
%  \ead{hou@cms.caltech.edu}
%  \address[cms]{Computing and Mathematical Sciences, California Institute of Technology, Pasadena, CA 91125}
  \cortext[cor1]{Corresponding author}
%  \cortext[cor2]{Co-first author}

\begin{abstract}
%In \cite{YauYau:2008}, Yau and Yau developed a novel algorithm to solve the path-wise robust DMZ equation, which laid down a solid foundation to study the nonlinear filtering problem (NLF).  Since then many efforts have been devoted to solve the path-wise robust DMZ equation. However, it is still very challenging to solve high-dimensional problem in an effective fashion. 
In this paper, we propose an efficient numerical method to solve high-dimensional nonlinear filtering (NLF) problems. Specifically, we use the tensor train decomposition method to solve the forward Kolmogorov equation (FKE) arising from the NLF problem. Our method consists of offline and online stages. In the offline stage, we use the finite difference method to discretize the partial differential operators involved in the FKE and extract low-dimensional structures in the solution space using the tensor train decomposition method. In addition, we approximate the evolution of the FKE operator using the tensor train decomposition method.  In the online stage using the pre-computed low-rank approximation tensors, we can quickly solve the FKE given new observation data. Therefore, we can solve the NLF problem in a real-time manner. Under some mild assumptions, we provide convergence analysis for the proposed method. Finally, we present numerical results to show the efficiency and accuracy of the proposed method in solving high-dimensional NLF problems.   

%	15A69  	Multilinear algebra, tensor products
%   35R60  	Partial differential equations with randomness, stochastic partial differential equations
%   60G35  	Signal detection and filtering
%	65Mxx	Partial differential equations, initial value and time-dependent initial-boundary value problems
%   65M12  	Stability and convergence of numerical methods
%   65M99  	None of the above, but in this section (Partial differential equations, initial value and time-dependent initial-boundary value problems)
\end{abstract}
\end{frontmatter}

% \newpage
% \tableofcontents
% \listoffigures
% \listoftables
% \newpage

\section{Introduction} \label{sec:introduction}
\noindent
Nonlinear filtering (NLF) problem is originated from the problem of  tracking and signal processing. The fundamental problem in the NLF is to give the instantaneous and accurate estimation of the states based on the noisy observations \cite{kallianpur:2013}. In this paper, we consider the signal based nonlinear filtering problems as follows,
\begin{align}\label{NonlinearSignalModelIntroSection}
\begin{cases}
d\textbf{x}_{t} = \textbf{f}(\textbf{x}_t,t)dt + \textbf{g}(\textbf{x}_t,t)d\textbf{v}_t, \\
d\textbf{y}_{t} = \textbf{h}(\textbf{x}_t,t)dt +         d\textbf{w}_t,
\end{cases} 
\end{align}
where $\textbf{x}_t\in R^d$ is the vector of states at time $t$, $\textbf{y}_t\in R^m$ is the vector of observations at time $t$, $\textbf{f}(\textbf{x}_t,t)=(f_1(\textbf{x}_t,t),...,f_d(\textbf{x}_t,t))^{T}$, $\textbf{g}(\textbf{x}_t,t)=(g_1(\textbf{x}_t,t),...,g_d(\textbf{x}_t,t))^{T}$, 
$\textbf{h}(\textbf{x}_t,t)=(h_1(\textbf{x}_t,t),...,h_m(\textbf{x}_t,t))^{T}$ are the drift, diffusion, and observation vector functions respectively, and $\textbf{v}_t$, $\textbf{w}_t$ are vector Brownian motion processes. Some growth conditions on $\textbf{f}$, $\textbf{g}$ and $\textbf{h}$ are required to guarantee the existence and well-posedness of the NLF problems, which will be discussed later. 

Particle filter method is the most popular method to solve \eqref{NonlinearSignalModelIntroSection}; see e.g. \cite{Platen:1992,kurtz1999particle,Arulampalam:2002,Gustafsson:2002,Bain:2009} and references therein. However, the main drawback of the particle filter method is that it is hard to be implemented in a real-time manner due to its nature of the Monte Carlo simulation. In practice, the real-time manner means the running time of the numerical 
integrator in solving the state equation for $\textbf{x}_t$ is much less than the time between any two observations of $\textbf{y}_{t}$.

Alternatively, one can solve the Duncan-Mortensen-Zakai (DMZ) equation, also known as Zakai equation, to study the NLF problems \cite{Duncan:1967,Mortensen:1966,Zakai:1969}. The DMZ equation computed the unnormalized conditional density function of the states $\textbf{x}_{t}$, which provides a powerful tool to study the NLF problem since one can estimate the statistical quantities of the state $\textbf{x}_{t}$ based on the DMZ solution. In general, one cannot solve the DMZ equation analytically. Many efforts have been made to develop efficient numerical methods; see e.g.\cite{kushner1967,Bensoussan:1990,Nagase:1995,ItoZakai:1996,RozovskiiLototsky:1997,zhang1998nonlinear,Gyongy:2003} 
and the references therein. 
 
The DMZ equation allows one to study the statistical quantities of the states $\textbf{x}_{t}$. In practice, however, one can only get one realization of the states $\textbf{x}_{t}$ (instead of thousands of repeated experiments), which motivates researchers to develop robust methods in solving the DMZ equation. Namely, the robust method should not be sensitive to the given observation paths. A novel algorithm was proposed to solve the path-wise robust DMZ equation \cite{YauYau:2008}.  In this approach, for each realization of the observation process denoted by $\textbf{y}_t$, one can make an invertible exponential transformation and transform the DMZ equation into a deterministic partial differential equation (PDE) with stochastic coefficient. Several efficient numerical methods were developed along this direction;
see \cite{LuoYauCompleteRealTime:2013,LuoYau:2013,WangLuoYauZhang:2018}, which can be efficient when the dimension of the NLF problems is small. However, it becomes expensive as the dimension of the NLF problems increases. Therefore, it is still very challenging to solve high-dimensional NLF problems in a real-time manner.
  
In this paper, we propose to use the tensor train decomposition method to solve the high-dimensional FKEs. Our method consists of offline and online stages. In the offline stage, we use the finite difference method to discretize the partial differential operators involved in the FKE and extract low-dimensional structures in the solution space using the tensor train decomposition method. Moreover, we approximate the evolution of the FKE operator using the tensor train method.  In the online stage, we can quickly solve the FKE given new observation data using the pre-computed low-rank approximation tensors. By exploring the low-dimensional structures in the solution space of the FKE, we can solve the NLF problem in a real-time manner. We also provide convergence analysis for the proposed method. Finally, we present numerical results to show the efficiency and accuracy of the proposed method. We find that the tensor train method is scalable in solving the FKE. Thus, we can solve the NLF problem in a real-time manner.  
%since the computational time grows logarithmically with respect to the degree of    

The rest of the paper is organized as follows. In Section 2, we give a brief introduction of the NLF problem
and DMZ  equation. In Section 3, the basic idea of the tensor train decomposition method is introduced. 
In Section 4, we propose our fast method to compute the high-dimensional FKEs. %More details about the offline and online computing will be discussed, including the computational complexity. 
Some convergence analysis of the proposed method will be provided in Section 5. In Section 6, we present numerical results to demonstrate the accuracy and efficiency of our method. Concluding remarks are made in Section 7.

\section{Some basic results of the NLF problems and DMZ equation}
\noindent
To start with, we consider the signal based model as follows,
\begin{align}\label{NonlinearSignalModel}
\begin{cases}
d\textbf{x}_{t} = \textbf{f}(\textbf{x}_t,t)dt + \textbf{g}(\textbf{x}_t,t)d\textbf{v}_t, \\
d\textbf{y}_{t} = \textbf{h}(\textbf{x}_t,t)dt +         d\textbf{w}_t,
\end{cases}
\end{align}
where $\textbf{x}_t\in R^d$ is a vector of the states of the system at time $t$, the initial state $\textbf{x}_0$ is given, $\textbf{y}_t\in R^m$ is a vector of the observations at time $t$ with initial state $\textbf{y}_0$, and $\textbf{v}_t$ and $\textbf{w}_t$ are vector Brownian motion processes with covariance matrices $E[d\textbf{v}_td\textbf{v}_t^{T}]=\textbf{Q}(t)dt\in R^{d\times d}$ and $E[d\textbf{w}_td\textbf{w}_t^{T}]=\textbf{S}(t)dt\in R^{m\times m}$, respectively. Moreover, $\textbf{x}_0$, $d\textbf{w}_t$ and $d\textbf{v}_t$ are assumed to be independent. Some growth conditions on $\textbf{f}$ and $\textbf{h}$ are required to guarantee the existence and uniqueness of the pathwise-robust DMZ equation \cite{YauYau:2008}. In this paper, $\textbf{f}$, $\textbf{h}$, and $\textbf{g}$ are $C^2$ in spatial variable and $C^1$ in the temporal variable.

%The most popular method so far to solve \eqref{NonlinearSignalModel} is the particle filter, see \cite{Arulampalam:2002,Gustafsson:2002,Bain:2009} and references therein. However, the main drawback of the particle filter method is that it is hard to be implemented in a real-time manner due to its nature of the Monte Carlo simulation.

The DMZ equation or Zakai equation \cite{Duncan:1967,Mortensen:1966,Zakai:1969} asserts that the unnormalized conditional density function of the states $\textbf{x}_{t}$, denoted by $\sigma(\textbf{x},t)$,  satisfies the following stochastic partial differential equation (SPDE):
\begin{align}\label{DMZequation}
\begin{cases}
d\sigma(\textbf{x},t) = \mathcal{L}\sigma(\textbf{x},t)dt + \sigma(\textbf{x},t)\textbf{h}^{T}(\textbf{x},t)\textbf{S}^{-1}d\textbf{y}_t, \\
\sigma(\textbf{x},0) = \sigma_0(\textbf{x}),
\end{cases}
\end{align}
where $\sigma_0(\textbf{x})$ is the density of the initial states $\textbf{x}_0$, and
\begin{align}\label{L_operator}
\mathcal{L}(\cdot):=\frac{1}{2}\sum_{i,j=1}^{d}\frac{\partial^2}{\partial x_i\partial x_j}
\big( (gQg^T)_{ij}\cdot\big) - \sum_{i=1}^{d}\frac{\partial (f_i \cdot)}{\partial x_i}.
\end{align}
The DMZ equation laid down a solid foundation to study the NLF problem. However, one cannot solve the DMZ equation analytically in general. Many efforts have been made to develop efficient numerical methods. One of the commonly used method is the splitting-up method originated from the Trotter product formula, which was first introduced in \cite{Bensoussan:1990} and has been extensively studied later, see \cite{Nagase:1995,ItoZakai:1996,Gyongy:2003}. In \cite{RozovskiiLototsky:1997}, the so-called $S^3$ algorithm was developed based on the Wiener chaos expansion. By separating the computations involving the observations from those dealing only with the system parameters, this approach gives rise to a new numerical scheme for NLF problems. However, the limitation of their method is that the drifting term $f$ and the observation term $h$ in \eqref{NonlinearSignalModel} should be bounded.
 
To overcome this restriction, Yau and Yau \cite{YauYau:2008} developed a novel algorithm to solve the path-wise robust DMZ equation. Specifically, for each realization of the observation process denoted by $y_t$, they make an invertible exponential transformation
\begin{align}\label{InvertibleTransformation}
\sigma(\textbf{x},t) = \exp\big( \textbf{h}^{T}(\textbf{x},t)\textbf{S}^{-1}(t)\textbf{y}_t \big)u(\textbf{x},t),
\end{align}
and transform the DMZ equation \eqref{DMZequation} into a deterministic partial differential equation (PDE) with stochastic coefficient
\begin{equation}\label{Intro_PathwiseRobustDMZ}
\left\{\begin{aligned}
\frac{\partial }{\partial t}u(\textbf{x},t)&+ \frac{\partial }{\partial t} (\textbf{h}^T\textbf{S}^{-1})\textbf{y}_t u(\textbf{x},t) = \\
&\exp\big(-\textbf{h}^{T}(\textbf{x},t)\textbf{S}^{-1}(t)\textbf{y}_t \big)\big(\mathcal{L}-\frac{1}{2}\textbf{h}^T\textbf{S}^{-1}\textbf{h}\big)\big(\exp\big(-\textbf{h}^{T}(\textbf{x},t)\textbf{S}^{-1}(t)\textbf{y}_t \big) u(\textbf{x},t)\big),\\
u(\textbf{x},0) = &\sigma_0(\textbf{x}).
\end{aligned}\right.
\end{equation}
Equation \eqref{Intro_PathwiseRobustDMZ} is called the pathwise robust DMZ equation \cite{RozovskiiLototsky:1997,YauYau:2008}.  Compared with the DMZ equation \eqref{DMZequation}, 
the pathwise robust DMZ equation \eqref{Intro_PathwiseRobustDMZ} is easier to solve, since the stochastic term has been transformed into the coefficients. 

The existence and uniqueness of \eqref{Intro_PathwiseRobustDMZ} has been investigated by many researchers. The well-posedness is guaranteed when the drift term $\textbf{f}\in C^1$ and the observation term $\textbf{h}\in C^2$ are bounded in \cite{Pardoux:1980}. Later on, similar results can be obtained under weaker conditions. For instance, the well-posedness results on the pathwise-robust DMZ equation with a class of unbounded coefficients were obtained in \cite{BarasZakai:1983,fleming:1982}, but the results were for one-dimensional case. In \cite{YauYau:2008}, the third author of this paper and his collaborator established the well-posedness result under the condition that $\textbf{f}$ and $\textbf{h}$ have at most linear growth.  In \cite{LuoYauCompleteRealTime:2013}, a well-posedness result was obtained for time-dependent pathwise-robust DMZ equation under some mild growth conditions on $\textbf{f}$ and $\textbf{h}$.

Let us assume that the  observation time sequences $0=t_0 < t_1 < \cdot\cdot\cdot < t_{N_t} = T$ are given. In each time interval $t_{j-1}\leq t < t_{j}$, one freezes the stochastic coefficient $\textbf{y}_{t}$ to be $\textbf{y}_{t_{j-1}}$ in Eq.\eqref{Intro_PathwiseRobustDMZ} and makes the exponential transformation
\begin{align}\label{PathwiseRobustDMZ_expotransf}
\widetilde{u}_j(\textbf{x},t) = \exp\big( \textbf{h}^{T}(\textbf{x},t)\textbf{S}^{-1}(t)\textbf{y}_{t_{j-1}} \big)u(\textbf{x},t).
\end{align}
It is easy to derive that $\widetilde{u}_j$ satisfies the FKE
\begin{align}\label{PathwiseRobustDMZ_KFE}
\frac{\partial }{\partial t}\widetilde{u}_{j}(\textbf{x},t) = \big(\mathcal{L}-\frac{1}{2}\textbf{h}^T\textbf{S}^{-1}\textbf{h}\big) \widetilde{u}_{j}(\textbf{x},t),
\end{align}
where the operator $\mathcal{L}$ is defined in \eqref{L_operator}.
In \cite{LuoYau:2013}, Luo and Yau investigated the Hermite spectral method to numerically solve the 1D FKE \eqref{PathwiseRobustDMZ_KFE} and analyzed the convergence rate of the proposed method. In their algorithm, the main idea is to shift part of the heavy computations off-line, so that only computations involved observations are performed on-line and synchronized with off-line data. The numerical method 
based on the Hermite polynomial approximation is efficient though, it is extremely hard to extend to solve high-dimensional FKEs in the real-time manner, since the number of the Hermite polynomial basis functions grows fast for high-dimensional problems. Namely, it suffers from the curse of dimensionality.

In a very recent result \cite{WangLuoYauZhang:2018}, we proposed to use the proper orthogonal decomposition (POD) method to numerically solve the 2D FKE. By extracting the low-dimensional structures in the solution space of the FKE and building POD basis, our method provides considerable savings over the Hermite polynomial approximation method that was used in \cite{LuoYau:2013}. The POD method helps us alleviate the curse of dimensionality to a certain extent though, it is still very challenging to solve high-dimensional NFL problems. The reason is that in the POD method one needs to compute solution snapshots to construct POD basis. However, to compute solution snapshots for high-dimensional NFL problems is extremely expensive. We shall address this challenge by using the tensor train  (TT)  decomposition method in this paper.

%In addition, they investigate the Hermite spectral method to numerical solve the 1-D forward Kolmogorov equation and analyze the convergence rate of the proposed method in \cite{LuoYau:2013}. However, it is still very challenging to solve high-dimensional forward Kolmogorov equation in a real-time fashion.
\section{The Tensor Train decomposition method}\label{sec:TTmethod}
\noindent
We shall introduce the tensor train decomposition method for approximating solutions of high-dimensional NLF problems. Let us assume the dimension of the NLF problem is $d$. 
For any fixed time $t$, if we discretize the solution $\tilde{u}_j(\textbf{x},t)$, $\textbf{x}\in R^d$ of the FKE \eqref{PathwiseRobustDMZ_KFE} using conventional numerical methods, such as finite difference method,  we obtain a $d$-dimensional $n_1\times n_2\times\cdots \times n_d$ tensor $\textbf{U}(i_1,i_2,\cdots,i_d)$, which is a multidimensional array. The number of unknowns in this representation grows fast as $d$ increases and is subject to the curse of dimensionality. To attack this challenge, one should extract potential low-dimensional structures in the tensor and approximate the tensor in a certain data-sparse way.    

The TT decomposition method is an efficient method for tensor approximation \cite{Oseledets:2011}. A brief introduction of the TT-format is given below. If a $d$-dimensional $n_1\times n_2\times\cdots \times n_d$ tensor $\textbf{U}(i_1,i_2,\cdots,i_d)$ can be written as the element-wise form 
\begin{align}\label{TTformat}
\textbf{U}(i_1,i_2,\cdots,i_d) = \textbf{G}_1(i_1)\textbf{G}_2(i_2)\cdots \textbf{G}_d(i_d), \quad 1\leq i_k \leq n_k, 
\end{align}
where $\textbf{G}_k(i_k)$ is a $r_{k-1}\times r_k$ matrix for each fixed $i_k$, $1\leq k\leq d$ 
and $r_0=r_{d}=1$. We call the tensor $\textbf{U}$ is in the TT-format, if it is represented in the form \eqref{TTformat}. Furthermore, each element $\textbf{G}_k$ can be regarded as an 3-dimensional tensor of the size $r_{k-1}\times n_k\times r_k$. In the representation \eqref{TTformat}, $\textbf{G}_1,\textbf{G}_2,\cdots,\textbf{G}_d$ are called the cores of the TT-format tensor $\textbf{U}$, numbers $r_k$ are called TT-ranks, and numbers $n_1,n_2,\cdots,n_d$ are called mode sizes. With these definitions, the representation \eqref{TTformat} can be rewritten as  
\begin{align}
\textbf{U}(i_1,i_2,\cdots,i_d) = \sum_{\alpha_1,\cdots,\alpha_{d-1}}\textbf{G}_1(\alpha_0,i_1,\alpha_1)\textbf{G}_2(\alpha_1,i_2,\alpha_2)\cdots \textbf{G}_d(\alpha_{d-1},i_d,\alpha_d)
\end{align}
where $\alpha_0=\alpha_d=1, 1\leq\alpha_k\leq r_k$ for $1\leq k\leq d-1$. 
In practice, one only needs to store all the cores $\textbf{G}_k$ in the TT-format, in order to save a tensor.  Thus, if all the TT-ranks $r_k$ are bounded by a constant $r$ and the mode sizes $n_k$ are bounded by $N$, the storage of the $d$-dimensional tensor $\textbf{U}$ is $\mathcal{O}(dNr^2)$ in the TT-format. Recall that the storage of the tensor $\textbf{U}$ is about $\mathcal{O}(N^d)$, if no approximation is used.

To further reduce the storage of the TT-format, a quantized tensor train (QTT) format was introduced in \cite{Oseledets:2010,QTTofFunction:2011,QTTLaplace:2012,FokkerPlanck:2012}. The QTT format is derived by introducing virtual dimensions along each real dimension of a tensor. Specifically, suppose each one-dimensional size of the tensor $\textbf{U}$ is a power of 2, i.e. $n_1=n_2=\cdots=n_d=2^L$. The $d$-dimensional tensor $\textbf{U}$ can be reshaped to a $D$-dimensional tensor with $D=dL$, while each mode size is equal to 2. The QTT-format is the TT-format of the reshaped tensor, which has a larger number of dimension but much smaller mode sizes (here is 2) than the TT-format. The concepts of cores, QTT-ranks and mode sizes (all are equal to 2) of QTT-format are defined similarly as the TT-format. The storage of the QTT-format is further reduced to $\mathcal{O}(d\log_{2}(N)r^2)$.

One can also simply formulate the TT-format of $d$-dimensional matrices \cite{QTTLaplace:2012}\cite{Oseledets:2010}. The TT-format of a $d$-dimensional $(m_1\times\cdots\times m_d)\times(n_1\times\cdots\times n_d)$ matrix $\textbf{B}$ can be written as 
\begin{align}
\textbf{B}(i_1,\cdots,i_d,j_1,\cdots,j_d) = \sum_{\alpha_1=1}^{r_1} \cdots \sum_{\alpha_{d-1}=1}^{r_{d-1}}\textbf{G}_1(1,i_1,j_1,\alpha_1)\textbf{G}_2(\alpha_1,i_2,j_2,\alpha_2)\cdots \textbf{G}_d(\alpha_{d-1},i_d,j_d,1),
\end{align}
where the index 1 in $\textbf{G}_1$ and $\textbf{G}_d$ is due to the TT-ranks $\alpha_0=\alpha_d=1$. The definitions of cores, ranks and mode sizes are similar as tensors. The QTT representation of matrices will be used in this paper.

Simple calculations show that in the TT-format the computational complexity of 
addition (together with TT-rounding after addition) is $\mathcal{O}(dNr^3)$, matrix-by-vector multiplication is $\mathcal{O}(dN^2r^4)$, and Hadamard product is $\mathcal{O}(dNr^4)$. Therefore, the TT/QTT-format allows lower complexity of algebraic operation than dense matrix or tensor form. In practice, especially in solving time-evolution problems, one needs to apply the TT-rounding procedure in the computation. The computational complexity of the TT-rounding is $\mathcal{O}(dNr^3)$. The purpose of TT-rounding is to decrease TT/QTT-rank of a matrix or tensor already in the TT/QTT-format while preserving a given accuracy $\epsilon$. The QTT-format can further reduce the factor $N$ of complexity to $\log_{2}(N)$.

\section{The fast algorithm and its implementation}\label{sec:fastalg}
\noindent  
In this section, we present the fast algorithm to solve the NLF problem. Let us assume that the observation time sequences $0=t_0 < t_1 < \cdot\cdot\cdot < t_{N_t} = T$ are given. But the observation data $\{\textbf{y}_{t_j}\}$ at each observation time $t_j$, $j=0,...,N_t$ are unknown until the on-line experiment runs. For simplicity, we assume $t_{j}-t_{j-1}=\Delta T$. We shall study how to solve the FKE \eqref{PathwiseRobustDMZ_KFE} within each time interval $[t_{j-1},t_{j}]$ and compute the exponential transformation \eqref{PathwiseRobustDMZ_expotransf} using the QTT-format. If we use an explicit scheme to discretize the time derivative in the FKE \eqref{PathwiseRobustDMZ_KFE}, we get an  semi-discrete scheme as follows, 
\begin{align}\label{timestepping_FKE}
\widetilde{u}_{j}^n(\textbf{x}) = \Big\{\tau (\mathcal{L}-\frac{1}{2}\textbf{h}^T\textbf{S}^{-1}\textbf{h}) + I\Big\}\widetilde{u}_{j}^{n-1}(\textbf{x}), ~~n = 1,\cdots,\frac{\Delta T}{\tau},
\end{align}
where $\tau$ is the time step in discretizing the FKE \eqref{PathwiseRobustDMZ_KFE} and $\Delta T$ is the time interval between two sequential observations. 

Our algorithm consists of an offline procedure and an online procedure. In the offline procedure, we
compute the matrices of spatial discretization of operators in \eqref{timestepping_FKE} and convert them into the QTT-format in advance, which will significantly reduce the computational time in the online procedure.  In the online procedure,  we solve the FKE and update the solution with new observation data. We shall show that the online procedure can be accomplished in a real time manner since all the computations are done in the QTT-format.

\subsection{Spatial discretization and low-rank approximation}\label{SpatialDiscretization}
\noindent
We shall discuss how to discretize the FKE \eqref{PathwiseRobustDMZ_KFE} and represent the spatial discretization in the QTT-format. To simplify the notations and illustrate the main idea of our algorithm, we choose a squared domain and uniform grid. Specifically, we take a large $a>0$ and let $\Omega=[-a,a]^d$ denote the physical domain of the FKE. A  uniform grid is set on each dimension with $N=2^L$ points and mesh size $h=\frac{2a}{(N-1)}$. We use $x_k(l_k)$,  $l_k=1,\cdots,N$, $k=1,...,d$ to record the coordinates of grid points on $k$-th dimension, which can be 
written in a compact form $\textbf{x}_\textbf{l}=(x_{1}(l_1),...,x_{d}(l_d))$ with $\textbf{l}=(l_1,...,l_d)$. 
We also assume $\textbf{S} = \textbf{I}_d$, $\textbf{Q} = \textbf{I}_d$ and $\textbf{g} = \textbf{I}_d$ for simplicity of notation.

First, the Laplace operator in \eqref{PathwiseRobustDMZ_KFE} is discretized by using a finite difference (FD) scheme. The resulting matrix is of the form
\begin{align}\label{ndLaplace}
\Delta_{d} = \Delta_1\otimes\cdots\otimes \textbf{I} +\cdots + \textbf{I}\otimes\cdots\otimes\Delta_1,
\end{align}
where
\[
\Delta_1 = \frac{1}{h^2}\mathrm{tridiag}(1,-2,1),
\]
and $\textbf{I}$ is an $N\times N$ identity matrix. In particular, when $d=3$ the Laplace operator has the form
\begin{align}\label{3dLaplace}
\Delta_{3} = \Delta_1\otimes \textbf{I}\otimes \textbf{I} +\textbf{I}\otimes\Delta_1\otimes \textbf{I} + \textbf{I}\otimes \textbf{I}\otimes\Delta_1.
\end{align}
The QTT-format of matrix $\Delta_d$ has a low-rank representation that is bounded by 4 (see Corollary 5.3 
of \cite{QTTLaplace:2012}). 

Second, the convection operator $\sum_{i=1}^{d}\frac{\partial (f_i \cdot)}{\partial x_i}$ in \eqref{PathwiseRobustDMZ_KFE} (see also Eq.\eqref{L_operator}) is discretized by using a central difference scheme. Thus, the corresponding $d$-dimensional matrix has the form
\begin{align}\label{driftoperator}
\textbf{C}_d = (\textbf{C}\otimes \textbf{I}\otimes\cdots \otimes \textbf{I})\textbf{F}_1 + (\textbf{I}\otimes \textbf{C}\otimes\cdots\otimes \textbf{I})\textbf{F}_2 + \cdots + (\textbf{I}\otimes \textbf{I}\otimes\cdots\otimes \textbf{C})\textbf{F}_d,
\end{align}
where $\textbf{F}_k$'s are diagonal matrices associated with the diagonalization of tensor discretization of the drift functions $f_k$, i.e.,
\[
\textbf{F}_k(l_1,l_2,\cdots,l_d,l_1,l_2,\cdots,l_d) = f_k(x_1(l_1),x_2(l_2),\cdots,x_d(l_d)),
\]
for all $1\leq k\leq d$, and $\textbf{C}$ is an one-dimensional central difference operator,
\[
\textbf{C} = \frac{1}{h}\mathrm{tridiag}(-\frac{1}{2},0,\frac{1}{2}).
\]
Under certain conditions for the drift terms $f_k$, the QTT-format of matrix $\textbf{C}_d$ has a bounded low-rank representation. We summarize the result into the following lemma and the proof can be found in \cite{FokkerPlanck:2012}.
\begin{proposition}%[see Lemma 3.1, \cite{FokkerPlanck:2012}]
Suppose that the QTT-ranks of the functions $f_k$ on a tensor grid are bounded by $r$. Then, the QTT-rank of the matrix $\textbf{C}_d$ in \eqref{driftoperator} is bounded by $5dr$.
\end{proposition}
%Finally, we shall discuss the low-rank approximation in the QTT-format. 
Although an exact TT-decomposition of any tensor is feasible \cite{Oseledets:2011}, it rarely has a low-rank structure. Therefore, one should apply TT-rounding procedures in order to decrease the TT/QTT-ranks while preserving a given accuracy $\epsilon$. Let us consider the drift terms $f_k$ as an example. In order to construct the QTT-format of functions $f_k$ with low ranks, one can use the TT-SVD algorithm \cite{TTcrossApprox:2010,Oseledets:2011}. QTT-ranks of these QTT-format tensors are guaranteed to be small when an  approximation error $\epsilon$ is prescribed in the sense of Frobenius norm. We list the result of the TT-SVD algorithm as follows.  
%\textcolor{red}{Do we need to provide more details of the TT-SVD algorithm?}
\begin{proposition}[Theorem 2.2 of \cite{TTcrossApprox:2010}]\label{lemma_of_TTapproximation}
For any tensor $\textbf{A}$ with size $n_1\times n_2\times\cdots\times n_d$, there exists a tensor $\textbf{B}$ in the TT-format with TT-ranks $r_k$ such that
\[
||\textbf{A} - \textbf{B}||_F \leq \sqrt{\sum_{k=1}^{d-1}\epsilon_k^2},
\]
where $\epsilon_k$ is the distance from $\textbf{A}_k$ to its best rank-$r_k$ approximation in the Frobenius norm, 
\[
\epsilon_k = \min_{\mathrm{rank}\textbf{C}\leqslant r_k}||\textbf{A}_k-\textbf{C}||_F,
\]
and $\textbf{A}_k$ is the $k$-th unfolding matrix of tensor $\textbf{A}$ 
\[
\textbf{A}_k = \mathrm{reshape}\left(\textbf{A},\prod_{s=1}^k n_s,\prod_{s=k+1}^d n_s
\right).
\]
\end{proposition}
\begin{remark}
The Prop.\ref{lemma_of_TTapproximation} allows us to control the accuracy and TT/QTT-ranks, when we compute the approximation of any tensor in the TT/QTT-format.
\end{remark}
Finally, the approximation of the function $\textbf{h}^T \textbf{h}$ in the FKE \eqref{PathwiseRobustDMZ_KFE} in the QTT-format can be obtained using the same approach as $f_k$ in the convection operator. Specifically, we discretize the function $\textbf{h}^T \textbf{h}$ on the spatial tensor grid and diagonalize it to a matrix denoted by $\textbf{Q}_d$, i.e.,
\[
\textbf{Q}_d(l_1,l_2,\cdots,l_d,l_1,l_2,\cdots,l_d) = (\textbf{h}^T \textbf{h})(x_1(l_1),x_2(l_2),\cdots,x_d(l_d)).
\]
Then, we approximate $\textbf{Q}_d$ by a low-rank QTT-format using the TT-SVD algorithm. 

\subsection{The offline procedure}
\noindent
In the offline procedure, we first assemble the discretization of the operators involved in the FKE, including the Laplace operator \eqref{ndLaplace}, the convection operator \eqref{driftoperator} and the multiplication operator $\textbf{Q}_d$ associated with the function $\textbf{h}^T \textbf{h}$, into a tensor $\textbf{A}$, i.e. 
\begin{align}
\textbf{A} = \frac{1}{2}\Delta_{d} - \textbf{C}_d -\frac{1}{2}\textbf{Q}_d.
\end{align}
In this paper, we assume the drift and observation functions are time-independent. 
Thus, the full discrete scheme for the FKE \eqref{PathwiseRobustDMZ_KFE} becomes
\begin{align}\label{AddingOperators_FKE}
\reallywidehat{U}_{\textbf{l},j}^{n} = \reallywidehat{\big(\tau \textbf{A}+ \textbf{I}\big)}\reallywidehat{U}_{\textbf{l},j}^{n-1}, ~~n = 1,\cdots,\frac{\Delta T}{\tau},
\end{align}
where $\reallywidehat{U}_{\textbf{l},j}^{n}$ is the QTT-format approximation 
of $\widetilde{u}_j(\textbf{x}_\textbf{l},t_{j-1}+n\tau)$, $\textbf{l}=(l_1,l_2,\cdots,l_d)$, $1\leq l_i \leq N$, $i=1,...,d$, i.e., the numerical solution 
of the FKE \eqref{PathwiseRobustDMZ_KFE} (see Section \ref{Sec:FDscheme}) and $\reallywidehat{\big(\tau \textbf{A}+ \textbf{I}\big)}$ is the QTT-format approximation of the tensor $\big(\tau \textbf{A}+ \textbf{I}\big)$.

Recall that the discretizations of the Laplace operator, the convection operator and the multiplication operator associated with the function $\textbf{h}^T \textbf{h}$ all have low-rank approximations. Moreover, addition of matrices or tensors in the QTT-format only causes addition of QTT-ranks. Therefore, the tensor $\big(\tau \textbf{A}+ \textbf{I}\big)$ has a low rank QTT-format approximation with a given maximal QTT-rank $r$ or with a certain given precision $\epsilon$ in the sense of Frobenius norm.  

Notice that in the NLF problem, there will be no observation available during the time period with length $\Delta T$. Thus, we directly compute the tensor $(\tau \textbf{A}+ \textbf{I})^{\frac{\Delta T}{\tau}}$ and approximate 
it in the QTT-format. Namely, from the scheme \eqref{AddingOperators_FKE} we obtain that 
\begin{align}\label{AssemblingOperators_FKE}
\reallywidehat{U}_{\textbf{l},j}^{\frac{\Delta T}{\tau}} = \reallywidehat{\big(\tau \textbf{A}+ \textbf{I}\big)^{\frac{\Delta T}{\tau}}}\reallywidehat{U}_{\textbf{l},j}^{0}.
\end{align}
where $\reallywidehat{\big(\tau \textbf{A}+ \textbf{I}\big)^{\frac{\Delta T}{\tau}}}$ is the QTT-format of the tensor $\big(\tau \textbf{A}+ \textbf{I}\big)^{\frac{\Delta T}{\tau}}$.
Exact addition of $\tau \textbf{A}$ and $\textbf{I}$ in the QTT format only increases the rank by one. However, exact multiplication of matrices in the QTT-format will lead to a significant growth of QTT-ranks. In our algorithm, we apply TT-rounding to control the growth of the QTT-rank caused by matrix-matrix multiplication, which can be easily achieved and maintain accuracy \cite{Oseledets:2011,FokkerPlanck:2012}. 
 
\subsection{The online procedure}
\noindent
In this section, we shall demonstrate that using the tensor train decomposition method and precomputed
low-rank approximation tensors, we can achieve fast computing in the online stage. 
%Suppose the time interval between two sequential observations is $\Delta T$, and total time of concern is $T$. Then the observation time sequence is $0=t_0 < t_1 < \cdots < t_{N_t} = T$, i.e. at every time $t_j,j=0,\cdots,N_t$, a new observation $y_{t_j}$ comes. Note that $y_{t_0}=0$ because there is no information at initial time.

At first, we set the initial probability density  according to initial state $\textbf{x}_0$ and solve the FKE
\eqref{PathwiseRobustDMZ_KFE} with such initial condition. At each observing time $t_j$, when a new observation $\textbf{y}_{t_j}$ arrives, we compute the exponential transformation \eqref{PathwiseRobustDMZ_expotransf}, in order to get the initial condition of the FKE \eqref{PathwiseRobustDMZ_KFE}. Then, we solve the FKE
\eqref{PathwiseRobustDMZ_KFE} by our algorithm (\ref{AssemblingOperators_FKE}). All of these operations are done in QTT-format, thus we need to do TT-rounding operation after both the exponential transformation and solving the FKE
\eqref{PathwiseRobustDMZ_KFE} using our algorithm (\ref{AssemblingOperators_FKE}).

\begin{proposition}\label{qtt_online_complexity}
	Suppose the QTT-ranks of all functions required in online procedure on a tensor grid, including $u(\textbf{x},t_j)$ and $\exp[\textbf{h}^T(\textbf{x},t_{j})(\textbf{y}_{t_{j}}-\textbf{y}_{t_{j-1}})]$, are bounded by $r$. The accuracy $\epsilon$ of TT-rounding is properly specified to ensure QTT-ranks of $u(\textbf{x},t)$ are also bounded by $r$ after any TT-rounding procedure. Then, the complexity of the online procedure within each time interval $[t_{j-1},t_{j}]$ is $\mathcal{O}(N^dr^2+d\log_{2}(N)r^6)$, where $N$ is the grid number on each dimension. 
\end{proposition}
\begin{proof}
The complexity of constructing the QTT-format of  $\exp[\textbf{h}^T(\textbf{x},t_{j})(\textbf{y}_{t_{j}}-\textbf{y}_{t_{j-1}})]$ from a full multidimensional array is $\mathcal{O}(N^dr^2)$ by Theorem 2.1 in \cite{Oseledets:2010}. The exponential transformation is essentially a Hadamard product in the QTT-format whose complexity is $\mathcal{O}(d\log_{2}(N)r^4)$ \cite{Oseledets:2011}. Solving the FKE \eqref{PathwiseRobustDMZ_KFE} is practically a matrix-vector multiplication (\ref{AssemblingOperators_FKE}) in the QTT-format whose complexity is $\mathcal{O}(d\log_{2}(N)r^4)$ \cite{Oseledets:2011}. Requirement of TT-rounding through standard TT-SVD algorithm is $\mathcal{O}(d\log_{2}(N)r^6)$ \cite{Oseledets:2011}.
\end{proof}
Notice that the total degree of freedom is $N^d$ in the spatial discretization. Prop.\ref{qtt_online_complexity} shows that the QTT method is very efficient in the online procedure in solve the NLF problem.  More details will 
be represented in Section \ref{sec:Numerical_example}. We observe that the  QTT-rank $r$ has very slow
growth with respect to $N$ (see Table \ref{tab:erank1}--Table \ref{tab:erank4}), which allows us to solve high-dimensional NLF problems in a real time manner.

%\begin{remark}
%A lower complexity of the QTT-format than the TT-format is obtained by introducing virtual dimensions along each real dimension.  \textcolor{red}{$d\log_2N$ is the dimension of the QTT-format}. The logarithmic dependence of QTT operations on the total degree of freedom $N^d$ is the key of computational improvement.
%\end{remark}
%\begin{remark}
%	Note that the term $N^dr^2$ is a result of construting a QTT-format tensor from a full tensor. In the case of 3D problems of our experiments, it has less effect on the whole computation than the term $d\log(N)r^6$. Actually the TT-rounding does cost most computational time as a dominating term. 
%	%In the case of 3-d problems of our experiments, the TT-rounding procedure is the most computationally expensive part. Constructing a QTT-format from a full tensor is the second costly part. In lower dimension Other parts have less effect on the whole online computation as our conclusion above. 
%	%In higher dimension case, the unavoidable linear dependence on the degree of freedom will become dominant.
%	
%\end{remark}

\subsection{The complete algorithm of the NLF problem}\label{CompleteAlgorithm}
\noindent
In this subsection, we give the complete algorithm of the NLF problem. The off- and on-line computing stages in our algorithm are summarized in the Algorithm \ref{alg:OfflineStageTT} and Algorithm \ref{alg:OnlineStageTT}, respectively. The performance of our method will be demonstrated in Section \ref{sec:Numerical_example}.
\begin{algorithm}[h]
	\caption{Offline computing}  \label{alg:OfflineStageTT}
	\begin{algorithmic}[1]
		\STATE Compute matrices of spatial discretization of operators mentioned in the Section \ref{SpatialDiscretization}, including the Laplace operator, i.e., Eq.\eqref{ndLaplace}, the convection operator, i.e., Eq.\eqref{driftoperator}, and the multiplication operator $\textbf{Q}_d$ associated with the function $\textbf{h}^T\textbf{S}^{-1}\textbf{h}$. 
		\STATE Convert these matrices  into the QTT-format. 
		%\textcolor{red}{??? Convert spatial grid tensor into the QTT-format in order to compute statistics required in online stage conveniently.???}
		\STATE Compute the addition of operator matrices in the QTT-format by taking into account the time step $\tau$, i.e. compute $\tau \textbf{A}+\textbf{I}$ in Eq.\eqref{AddingOperators_FKE}.
		\STATE Compute the power of the tensor $\tau \textbf{A}+\textbf{I}$ in the QTT-format, i.e. compute $\reallywidehat{(\tau \textbf{A}+\textbf{I})^{\frac{\Delta T}{\tau}}}$ in Eq.\eqref{AssemblingOperators_FKE}.
	\end{algorithmic}
	%\footnotetext[1]{If \eqref{DMZequation} is time-invariant and the observation interval is uniform, we only need to calculate the propagator once and denote by $I^{\Delta t}\varphi_k$.}
\end{algorithm}	

\begin{algorithm}[h]
	\caption{Online computing}  \label{alg:OnlineStageTT}
	\begin{algorithmic}[1]
	\STATE Set up the initial data $u(\textbf{x},0) = \sigma_0(\textbf{x})$ of the FKE \eqref{PathwiseRobustDMZ_KFE}  according to the distribution of the initial state $\textbf{x}_0$, convert  $u(\textbf{x},0)$ into a QTT-format, and apply the propagator operator \eqref{AssemblingOperators_FKE} to get the predicted solution at time $t_1$, denoted by $\reallywidehat{U}_{\textbf{l},1}^{\frac{\Delta T}{\tau}}$. 
	\FOR{$j=1\to N_t-1$}
	\STATE Convert the term $\exp[\textbf{h}^T(\textbf{x},t_{j})\textbf{S}^{-1}(t_{j})(\textbf{y}_{t_{j}}-\textbf{y}_{t_{j-1}})]$ into the QTT-format.
	\STATE Assimilate the new observation data $\textbf{y}_{t_j}$ into the predicted solution $\reallywidehat{U}_{\textbf{l},j}^{\frac{\Delta T}{\tau}}$ using a QTT-format Hadamard product:
	\[		
	\reallywidehat{U}_{\textbf{l},j+1}^{0}=\reallywidehat{\exp[\textbf{h}^T(\textbf{x},t_{j})\textbf{S}^{-1}(t_{j})(\textbf{y}_{t_{j}}-\textbf{y}_{t_{j-1}})]}\reallywidehat{U}_{\textbf{l},j}^{\frac{\Delta T}{\tau}}.
	\]
	\STATE Compute the predicted solution at time $t_{j+1}$ using a matrix-vector multiplication in the QTT-format:
	\[
	\reallywidehat{U}_{\textbf{l},j+1}^{\frac{\Delta T}{\tau}} = \reallywidehat{(\tau \textbf{A}+ \textbf{I})^{\frac{\Delta T}{\tau}}}\reallywidehat{U}_{\textbf{l},j+1}^{0}.
	\]
	\STATE Calculate related statistics of prediction by using $\reallywidehat{U}_{\textbf{l},j+1}^{\frac{\Delta T}{\tau}}$ as the unnormalized density function at time $t_{j+1}$.
	\ENDFOR
		%\STATE Calculate the final statistics by using $u_{N_t+1}(\cdot,t_{N_t})$ as the final UFD.
	\end{algorithmic}
\end{algorithm}

\section{Convergence analysis}
\noindent
In this section, we shall study the convergence of the numerical solution obtained by our method to the solution 
of the DMZ equation. For simplicity of notations in the analysis, we assume $\textbf{S} = \textbf{I}_d$, $\textbf{Q} = \textbf{I}_d$ and $\textbf{g} = \textbf{I}_d$. Note that the proof is straightforward if $\textbf{Q}$, $\textbf{S}$ are general covariance matrices and $\textbf{g}$ is a general matrix. 
\subsection{Some assumptions and propositions}
\noindent
Before proceeding to the main analysis, let us first introduce some assumptions as follows.
\begin{enumerate}
	\item[[Asm.1]] The following term is bounded in $\mathcal{R}^d\times [0,T]$, i.e., 		
\begin{align}
-\frac{1}{2}\textbf{h}^T\textbf{h}-\frac{1}{2}\Delta K - \textbf{f}\cdot\nabla K+\frac{1}{2}|\nabla K|^2 + |\textbf{f}-\nabla K |\leq c_1, \text{~~}\forall (\textbf{x},t)\in \mathcal{R}^d\times [0,T],
\label{condition-for-elliptic}
\end{align}
where $K = \textbf{h}^T \textbf{y}_t$, $c_1$ is a constant possibly depending on $T$.
\item[[Asm.2]] The drift function $\textbf{f}$ is bounded in a bounded domain $\Omega$, i.e.
$\sup|f_i(\textbf{x})| \leq C_f < \infty, \forall\textbf{x}\in\Omega,\text{~~~} i=1,2,\cdots,d$
and Lipschitz continuous, i.e. $|f_i(\textbf{x}_1) - f_i(\textbf{x}_2)| \leq L_f |\textbf{x}_1 - \textbf{x}_2|, \forall\textbf{x}_1,\textbf{x}_2\in\Omega,~i=1,2,\cdots,d$, where  $ L_f$ is the Lipschitz constant. 
\item[[Asm.3]] The observation function $\textbf{h}$ is bounded in a bounded domain $\Omega$, i.e.
$\sup|h_i(\textbf{x})| \leq C_h < \infty,\forall\textbf{x}\in\Omega, \text{~~~} i=1,2, \cdots,m.$
\item[[Asm.4]] The observation series $K = \textbf{h}^T \textbf{y}_t$ is bounded in a bounded domain $\Omega$ on the observation time sequence $0=t_0 < t_1 < \cdot\cdot\cdot < t_{N_t} = T$, i.e.
	\begin{align}
	| 2K |\leq c_2, \text{~~}\forall (\textbf{x},t)\in \Omega\times \{t_0,t_1,\cdots,t_{N_t}\}.
	\label{bound-K}
	\end{align}
\end{enumerate}  
After introducing necessary assumptions, we are in the position to proceed the convergence analysis. When the condition \eqref{condition-for-elliptic} in Asm.1 is satisfied, one can choose a bounded domain $\Omega$ large enough to capture almost all the density of the DMZ equation \eqref{DMZequation}, since \eqref{DMZequation} is essentially a parabolic-type PDE. Thus, we can restrict the DMZ equation \eqref{DMZequation} on the bounded domain 
$\Omega$. 

Let $u(\textbf{x},t)$ be the solution of the DMZ equation \eqref{DMZequation} restricted on $\Omega\times[0,T]$ satisfying 
\begin{equation}\label{TheoremCeq1}
\left\{
\begin{aligned}
\frac{\partial u}{\partial t}(\textbf{x},t) &= \frac{1}{2}\Delta u(\textbf{x},t) + \textbf{F}(\textbf{x},t)\cdot\nabla u(\textbf{x},t) + J(\textbf{x},t)u(\textbf{x},t),\\
u(\textbf{x},0)&=\sigma_0(\textbf{x}),\\
u(\textbf{x},t)|_{\partial\Omega}&=0,
\end{aligned}
\right.
\end{equation}
where $\textbf{F}=-\textbf{f}+\nabla K$, $K = \textbf{h}^T \textbf{y}_t$, and $J=-$div$\textbf{f}-\frac{1}{2}\textbf{h}^T\textbf{h}+\frac{1}{2}\Delta K-\textbf{f}\cdot\nabla K+\frac{1}{2}|\nabla K|^2$. 

%For any $0\leq t\leq T$, 
Let $\mathcal{P}_{N_t}=\{0=t_0<t_1<\cdots<t_{N_t}=T\}$ be a partition of $[0,T]$, where $t_j = \frac{jT}{N_t}$, $j=0,...,N_t$. Let $u_j(\textbf{x},t)$ be the solution of the following equation defined 
on $\Omega\times[t_{j-1},t_j]$,
\begin{equation}\label{TheoremCeq2}
\left\{
\begin{aligned}
\frac{\partial u_j}{\partial t}(\textbf{x},t) &= \frac{1}{2}\Delta u_j(\textbf{x},t) + \textbf{F}(\textbf{x},t_{j-1})\cdot\nabla u_j(\textbf{x},t) + J(\textbf{x},t_{j-1})u_j(\textbf{x},t),\\
u_j(\textbf{x},t_{j-1})&=u_{j-1}(\textbf{x},t_{j-1}),\\
u_j(\textbf{x},t)|_{\partial\Omega}&=0, 
\end{aligned}
\right.
\end{equation}
where we use the condition $u_0(\textbf{x},t)=\sigma_0(\textbf{x})$. Then, the restriction of the solution $u(\textbf{x},t)$ of \eqref{TheoremCeq1} on each domain $\Omega\times[t_{j-1},t_j]$ can be approximated by the solution $u_j(\textbf{x},t)$ of \eqref{TheoremCeq2}. Specifically, we have the following error estimate. 
\begin{proposition}[Theorem C of \cite{YauYau:2008}]\label{YauYau:theoremC}
Let $\Omega$ be a bounded domain in $\mathcal{R}^d$. Let $\textbf{F}:\Omega\times[0,T]\rightarrow\mathcal{R}^d$ be a family of vector fields that are $\mathcal{C}^\infty$ in $\textbf{x}$ and Holder continuous in $t$ with exponent $\alpha$ and  $J:\Omega\times[0,T]\rightarrow\mathcal{R}$ be a $\mathcal{C}^\infty$ function in $\textbf{x}$ and Holder continuous in $t$ with exponent $\alpha$ such that following properties are satisfied 
\begin{align}
&|\mathop{div} \textbf{F}(\textbf{x},t)| + 2|J(\textbf{x},t)| + |\textbf{F}(\textbf{x},t)| \leq c_3, \text{~ for~} (\textbf{x},t)\in\Omega\times[0,T],\\
&|\textbf{F}(\textbf{x},t) - \textbf{F}(\textbf{x},\bar{t})| + |\mathop{div}\textbf{F}(\textbf{x},t) - \mathop{div}\textbf{F}(\textbf{x},\bar{t})| + |J(\textbf{x},t) - J(\textbf{x},\bar{t})|\leq c_4|t-\bar{t}|^\alpha, \nonumber \\ 
&\text{~ for~} (\textbf{x},t),(\textbf{x},\bar{t})\in\Omega\times[0,T].
\end{align}
Then, we have the following estimate holds
\begin{align}\label{TheoremCestimate}
\int_\Omega\big|u-u_{N_t}\big|(\textbf{x},t)d\textbf{x}\leq \frac{2c_5}{\alpha+1}\frac{T^{\alpha+1}e^{c_3T}}{{N_t}^\alpha},
\end{align} 
where $c_5 = c_4e^{c_3T} + c_4 \sqrt{\mathop{Vol}(\Omega)}e^{c_3^2T} \sqrt{2c_3^2T\int_\Omega u^2(\textbf{x},0) +\int_\Omega |\nabla u(\textbf{x},0)|^2 }$. Specifically, $u(\textbf{x},t)=\lim_{N_t\rightarrow\infty}u_{N_t}(\textbf{x},t)$ in the $L^1$ sense on $\Omega$. 
The convergence rate $\alpha$ depends on the regularity of  $\textbf{F}$ and $J$ in the time variable $t$. 
\end{proposition}
%We can choose $\textbf{F}=-\textbf{f}+\nabla K$ and $J=-$div$\textbf{f}-\frac{1}{2}\textbf{h}^T\textbf{h}+\frac{1}{2}\Delta K-\textbf{f}\cdot\nabla K+\frac{1}{2}|\nabla K|^2$  and let $\mathcal{P}_k$ to be observation time sequence in Theorem \ref{YauYau:theoremC}, where $K = \textbf{h}^T \textbf{y}_t$, such that \eqref{TheoremCeq1} and \eqref{TheoremCeq2} are DMZ equations. 
It is shown in \cite{YauYau:2008} (see Proposition 2.1) that the pathwise robust DMZ equations can be computed by solving the FKE \eqref{PathwiseRobustDMZ_KFE}. 
\begin{proposition}[Proposition 2.1 of \cite{YauYau:2008}]\label{YauYau:proposition2.1}
$\widetilde{u}_j(\textbf{x},t)$ satisfies the forward Kolmogorov equation
\begin{align}\label{FKEinYauProposition}
\frac{\partial \widetilde{u}_j}{\partial t}(\textbf{x},t)=\frac{1}{2}\Delta\widetilde{u}_j(\textbf{x},t)-\textbf{f}(\textbf{x})\cdot\nabla \widetilde{u}_j(\textbf{x},t) - \Big(\mathop{div} f(\textbf{x}) + \frac{1}{2}\textbf{h}(\textbf{x})^T\textbf{h}(\textbf{x})\Big)\widetilde{u}_j(\textbf{x},t)
\end{align}
for $t_{j-1}\leq t\leq t_j$ if and only if
\begin{align}
u_j(\textbf{x},t) = \exp(-\textbf{h}(\textbf{x})^T\textbf{y}_{t_{j-1}})\widetilde{u}_j(\textbf{x},t)
\label{u-equal-exputitle}
\end{align}
satisfies the robust DMZ equation with observation being frozen at $\textbf{y}_{t_{j-1}}$:
\begin{align}
\frac{\partial u_j}{\partial t}(\textbf{x},t) &= \frac{1}{2}\Delta u_j(\textbf{x},t) + \textbf{F}(\textbf{x},t_{j-1})\cdot\nabla u_j(\textbf{x},t) + J(\textbf{x},t_{j-1})u_j(\textbf{x},t),
\end{align}
where $\textbf{F}=-\textbf{f}+\nabla K$, $K = \textbf{h}^T \textbf{y}_t$, and $J=-$div $\textbf{f}-\frac{1}{2}\textbf{h}^T\textbf{h}+\frac{1}{2}\Delta K-\textbf{f}\cdot\nabla K+\frac{1}{2}|\nabla K|^2$.
\end{proposition}
In fact, the FKE \eqref{FKEinYauProposition} is obtained from the general form of the FKE \eqref{PathwiseRobustDMZ_KFE} by letting $\textbf{S} = \textbf{I}_d$, $\textbf{Q} = \textbf{I}_d$ and $\textbf{g} = \textbf{I}_d$.

\subsection{Convergence analysis for the FD scheme}\label{Sec:FDscheme}
\noindent
From time $t=t_{j-1}$ to time $t=t_{j}$, one can solve the FKE \eqref{FKEinYauProposition} by using the 
finite difference method. Specifically for time discretization, we partition the time interval $[t_{j-1},t_j]$ into an equispaced grid, i.e. $t_{j-1}+n\tau$, $n=0,1,\cdots,\frac{\Delta T}{\tau}$, where $\tau$ is the time step. We will analyze the error of FD scheme here. 
%To facilitate the convergence analysis, we rewrite the FKE \eqref{FKEinYauProposition} into an equivalent form 
%\begin{align}
%\frac{\partial }{\partial t}\widetilde{u}_{j}(\textbf{x},t) = \big(\frac{1}{2}\sum_{i_1,i_2=1}^{d}\frac{\partial^2}{\partial x_{i_1}\partial x_{i_2}}- \sum_{i=1}^{d}f_i\frac{\partial}{\partial x_i}- \sum_{i=1}^{d}\frac{\partial f_i}{\partial x_i}-\frac{1}{2}\textbf{h}^T\textbf{h}\big) \widetilde{u}_{j}(\textbf{x},t).
%\label{FKEinYauProposition-split}
%\end{align} 
The spatial discretization has been discussed in Section \ref{SpatialDiscretization}. We 
use the same notations here, i.e., $h$ is the spatial mesh size. Notice that $\textbf{h}$ is the observation function. Let $U_{\textbf{l},j}^{n}$ denote the approximations of the solution 
at these grid points and time $t=t_{j-1}+n\tau$, i.e., 
\begin{align}
U_{\textbf{l},j}^{n}\approx \widetilde{u}_{j}\big(\textbf{x}_{\textbf{l}},t_{j-1}+n\tau\big)=\widetilde{u}_{j}\big(x_1(l_1),x_2(l_2),\cdots,x_d(l_d),t_{j-1}+n\tau\big).
\end{align} 
Then, the FD scheme for the FKE \eqref{FKEinYauProposition} reads 
\begin{align}\label{FDscheme_for_FKE}
\frac{U^{n+1}_{\textbf{l},j} - U^n_{\textbf{l},j}}{\tau} &= \frac{1}{2}\frac{\sum_{i=1}^{d} (U^n_{\textbf{l}+\textbf{e}_i,j} + U^n_{\textbf{l}-\textbf{e}_i,j}) - (2d)U^n_{\textbf{l},j}}{h^2} - \frac{\sum_{i=1}^d (f_i)_{\textbf{l}}\big(  U^n_{\textbf{l}+\textbf{e}_i,j} -  U^n_{\textbf{l}-\textbf{e}_i,j}\big)  }{2h} \nonumber \\
&- \frac{\sum_{i=1}^d\big(  (f_i)_{\textbf{l}+\textbf{e}_i} -  (f_i)_{\textbf{l}-\textbf{e}_i}\big)  U^n_{\textbf{l},j} }{2h}- \frac{1}{2}(\textbf{h}^T\textbf{h})_\textbf{l} U^n_{\textbf{l},j} 
\end{align}
for all $\textbf{l}=(l_1,l_2,\cdots,l_d)$, $1\leq l_i \leq N$, $i=1,...,d$, which is a $d$-dimensional index vector and $\textbf{e}_i$ is a unit vector with a $1$ in the $i$th coordinate and $0$'s elsewhere. The convergence result for the FD scheme \eqref{FDscheme_for_FKE} can be summarized into the following two lemmas.  
\begin{lemma} 
The truncation error of the FD scheme \eqref{FDscheme_for_FKE} for the FKE \eqref{FKEinYauProposition} 
is $\mathcal{O}(\tau+h^2)$. Hence the FD scheme \eqref{FDscheme_for_FKE} is consistent.  
\end{lemma}
\begin{proof}
For each $\textbf{x}_{\textbf{l}}$, $\textbf{l}=(l_1,l_2,\cdots,l_d)$, $1\leq l_i \leq N$, $i=1,...,d$ and $t_j^n=t_{j-1}+n\tau$, $n=0,1,\cdots,\frac{\Delta T}{\tau}-1$, the second-order Taylor expansion of $\widetilde{u}_{j}(\textbf{x}_{\textbf{l}},t_j^n+\tau)$ at the point $(\textbf{x}_{\textbf{l}},t_j^n)$ gives
\begin{align}\label{truncationerror1}
\frac{\widetilde{u}_j(\textbf{x}_{\textbf{l}},t_j^n+\tau)-\widetilde{u}_j(\textbf{x}_{\textbf{l}},t_j^n)}{\tau} = \frac{\partial \widetilde{u}_j}{\partial t}(\textbf{x}_{\textbf{l}},t_j^n) + \frac{\tau}{2}\frac{\partial^2\widetilde{u}_j}{\partial t^2}(\textbf{x}_{\textbf{l}},t_j^n)+o(\tau).
\end{align}
The fourth-order Taylor expansion of $\widetilde{u}_{j}(\textbf{x}_{\textbf{l}}+h\textbf{e}_i,t_j^n)$ and $\widetilde{u}_{j}(\textbf{x}_{\textbf{l}}-h\textbf{e}_i,t_j^n)$ at the point $(\textbf{x}_{\textbf{l}},t_j^n)$, gives
\begin{align}\label{truncationerror2}
\frac{\widetilde{u}_j(\textbf{x}_{\textbf{l}}+h\textbf{e}_i,t_j^n) -2\widetilde{u}_j(\textbf{x}_{\textbf{l}},t_j^n) + \widetilde{u}_j(\textbf{x}_{\textbf{l}}-h\textbf{e}_i,t_j^n)}{2h^2} = \frac{1}{2}\frac{\partial^2\widetilde{u}_j}{\partial x_i^2}(\textbf{x}_{\textbf{l}},t_j^n) + \frac{h^2}{24}\frac{\partial^4\widetilde{u}_j}{\partial x_i^4}(\textbf{x}_{\textbf{l}},t_j^n) + o(h^2)
\end{align}
The third-order Taylor expansion of $(f_i\widetilde{u}_j)(\textbf{x}_{\textbf{l}}+h\textbf{e}_i,t_j^n)$ and $(f_i\widetilde{u}_j)(\textbf{x}_{\textbf{l}}-h\textbf{e}_i,t_j^n)$ at the point $(\textbf{x}_{\textbf{l}},t_j^n)$ gives
\begin{align}\label{truncationerror3}
\frac{(f_i\widetilde{u}_j)(\textbf{x}_{\textbf{l}}+h\textbf{e}_i,t_j^n) - (f_i\widetilde{u}_j)(\textbf{x}_{\textbf{l}}-h\textbf{e}_i,t_j^n)}{2h} = \frac{\partial(f_i\widetilde{u}_j)}{\partial x_i}(\textbf{x}_{\textbf{l}},t_j^n) + \frac{h^2}{6}\frac{\partial^3(f_i\widetilde{u}_j)}{\partial x_i^3}(\textbf{x}_{\textbf{l}},t_j^n) + o(h^2),
\end{align}
for $i=1,2,\cdots,d$. Combining \eqref{truncationerror1}, \eqref{truncationerror2} and \eqref{truncationerror3}, we get the truncation error of the FD scheme \eqref{FDscheme_for_FKE} as
\begin{align}
T^n_{\textbf{l},j} =  \frac{\tau}{2}\frac{\partial^2\widetilde{u}_j}{\partial t^2}(\textbf{x}_\textbf{l},t_j^n)+o(\tau) - \sum_{i=1}^d\frac{h^2}{24}\frac{\partial^4\widetilde{u}_j}{\partial x_i^4}(\textbf{x}_\textbf{l},t_j^n) - \sum_{i=1}^d\frac{h^2}{6}\frac{\partial^3(f_i\widetilde{u}_j)}{\partial x_i^3}(\textbf{x}_\textbf{l},t_j^n) + o(h^2) 
= \mathcal{O}(\tau+h^2).
\label{local-truncation-error}
\end{align} 
The consistency of the FD scheme is proved. 
\end{proof}
\begin{lemma}\label{lemma_of_FDconvergence}
Let $U_{\textbf{l},j}^{n}$, $n = 0,1,\cdots,\frac{\Delta T}{\tau}$ denote the numerical solution obtained by the FD scheme \eqref{FDscheme_for_FKE} and $\widetilde{u}_j(\textbf{x}_\textbf{l},t_{j-1}+n\tau)$ denote the corresponding exact solution to the FKE \eqref{FKEinYauProposition} on $[t_{j-1},t_j]$, 
respectively, where $t_j=t_{j-1}+\Delta T$. Suppose the assumptions Asm.1 to Asm.4 and the stability condition $h<\frac{1}{C_f}$,  $\tau<(\frac{d}{h^2}+dL_f+\frac{d}{2}C_h^2)^{-1}$ are satisfied. Then, we have the following error estimate for the FD scheme \eqref{FDscheme_for_FKE},
\begin{align}\label{FDtopathwiseDMZconvergence1}
||U_{\textbf{l},j}^{\frac{\Delta T}{\tau}} - \widetilde{u}_j(\textbf{x}_\textbf{l},t_{j})||_\infty \leq e^{\Delta TdL_f}||U_{\textbf{l},j}^{0} - \widetilde{u}_j(\textbf{x}_\textbf{l},t_{j-1})||_\infty+\frac{C}{dL_f}(\tau+h^2), 
~\forall \textbf{l}, %\in \{1,2,\cdots,N\}^d.
\end{align} 
where $C$ is a generic constant that does not depend on $\tau$ and $h$.
Finally, we have the estimate
%\begin{align}\label{FDtopathwiseDMZconvergence2}
%||u_{N_t}^{FD,-} - u_{N_t}(\cdot,T)||_\infty \leq \frac{Ce^{c_2N_t+TdL_f}}{dL_f(e^{c_2+\Delta TdL_f}-1)}(\tau+h^2),
%\end{align}
\begin{align}\label{FDtopathwiseDMZconvergence2}
||U_{\textbf{l},N_t}^{\frac{\Delta T}{\tau}} - \widetilde{u}_{N_t}(\textbf{x}_\textbf{l},T)||_\infty \leq \frac{Ce^{c_2N_t+TdL_f}}{dL_f(e^{c_2+\Delta TdL_f}-1)}(\tau+h^2),~\forall \textbf{l}, %\in \{1,2,\cdots,N\}^d.
\end{align} 
where $c_2$ is a constant defined in \eqref{bound-K} of Asm.4. %In particular, when $\tau,h\rightarrow 0$,  $||u_{N_t}^{FD,-} - u_{N_t}(\cdot,T)||_\infty\rightarrow 0$.
\end{lemma}

\begin{proof}
Let $e^n_{\textbf{l},j}= U_{\textbf{l},j}^{n} - \widetilde{u}_j(\textbf{x}_\textbf{l},,t_{j-1}+n\tau)$ denote the error between FD solution and exact solution of the FKE \eqref{FKEinYauProposition} on $[t_{j-1},t_j]$. By the definition of the truncation error, we have
\begin{align*}
\frac{e^{n+1}_{\textbf{l},j} - e^n_{\textbf{l},j}}{\tau} &= \frac{1}{2}\frac{\sum_{i=1}^{d} (e^n_{\textbf{l}+\textbf{e}_i,j} + e^n_{\textbf{l}-\textbf{e}_i,j}) - (2d)e^n_{\textbf{l},j}}{h^2}
- \frac{\sum_{i=1}^d (f_i)_{\textbf{l}}\big(  e^n_{\textbf{l}+\textbf{e}_i,j} -  e^n_{\textbf{l}-\textbf{e}_i,j}\big)  }{2h} \nonumber \\
&- \frac{\sum_{i=1}^d\big(  (f_i)_{\textbf{l}+\textbf{e}_i} -  (f_i)_{\textbf{l}-\textbf{e}_i}\big)  e^n_{\textbf{l},j}}{2h} - \frac{1}{2}(\textbf{h}^T\textbf{h})_\textbf{l} e^n_{\textbf{l},j} + T^n_{\textbf{l},j},
\end{align*}
i.e.
\begin{align}\label{FD_errorofonestep}
e^{n+1}_{\textbf{l},j} &= \Big(1-\frac{d\tau}{h^2} - \tau\sum_{i=1}^d\frac{\big(  (f_i)_{\textbf{l}+\textbf{e}_i} -  (f_i)_{\textbf{l}-\textbf{e}_i}\big)  }{2h}- \frac{\tau}{2}(\textbf{h}^T\textbf{h})_\textbf{l} \Big)e^{n}_{\textbf{l},j} \nonumber \\
&+ \sum_{i=1}^{d}  \Big(\frac{\tau}{2h^2}-\frac{\tau}{2h}(f_i)_{\textbf{l}}\Big)e^n_{\textbf{l}+\textbf{e}_i,j} + \sum_{i=1}^d \Big(\frac{\tau}{2h^2} + \frac{\tau}{2h}  (f_i)_{\textbf{l}}\Big) e^n_{\textbf{l}-\textbf{e}_i,j} + \tau T^n_{\textbf{l},j},
\end{align}
where $T^n_{\textbf{l},j}$ is the truncation error; see \eqref{local-truncation-error}. 
Let the maximum error at a time $t=t_{j-1}+n\tau$ denoted by
\begin{align}\label{FD_definitionofmaxerror}
	E_j^n := \max\{|e_{\textbf{l},j}^n|,  \textbf{l}=(l_1,l_2,\cdots,l_d), 1\leq l_i \leq N, i=1,...,d \}.
\end{align}
We introduce the stability conditions
\begin{align}
\frac{\tau}{2h^2}> \big|\frac{\tau}{2h}(f_i)_{\textbf{l}}\big|, \quad
1> \frac{d\tau}{h^2} + \tau\sum_{i=1}^d\frac{\big(  (f_i)_{\textbf{l}+\textbf{e}_i} -  (f_i)_{\textbf{l}-\textbf{e}_i}\big)  }{2h}+ \frac{\tau}{2}(\textbf{h}^T\textbf{h})_\textbf{l}.
\label{StabilityConditions}
\end{align}
According to Asm.2 and Asm.3, the stability conditions \eqref{StabilityConditions} require $h<\frac{1}{C_f}$ and $\tau<(\frac{d}{h^2}+dL_f+\frac{d}{2}C_h^2)^{-1}$, where $C_h$, $C_f$, $L_f$ are defined in Asm.2 and Asm.3, and $C$ is a generic constant in the truncation error. Under the 
stability conditions \eqref{StabilityConditions}, \eqref{FD_errorofonestep} implies that  
\begin{align}\label{E_jn+1n}
E_j^{n+1} \leq & \Big(1-\frac{d\tau}{h^2} - \tau\sum_{i=1}^d\frac{\big(  (f_i)_{\textbf{l}+\textbf{e}_i} -  (f_i)_{\textbf{l}-\textbf{e}_i}\big)  }{2h}- \frac{\tau}{2}(\textbf{h}^T\textbf{h})_\textbf{l} \Big) E_j^n \nonumber\\
& + \sum_{i=1}^{d}  \Big(\frac{\tau}{2h^2}-\frac{\tau}{2h}(f_i)_{\textbf{l}}\Big)E_j^n  + \sum_{i=1}^d \Big(\frac{\tau}{2h^2} + \frac{\tau}{2h}  (f_i)_{\textbf{l}}\Big) E_j^n  + \tau C(\tau+h^2)\nonumber\\
\leq & \Big(1 - \tau\sum_{i=1}^d\frac{\big(  (f_i)_{\textbf{l}+\textbf{e}_i} -  (f_i)_{\textbf{l}-\textbf{e}_i}\big)  }{2h}- \frac{\tau}{2}(\textbf{h}^T\textbf{h})_\textbf{l} \Big)E_j^n + \tau C(\tau+h^2),\nonumber\\
\leq & \Big(1 + \tau dL_f\Big)E_j^n + \tau C(\tau+h^2).
\end{align}
Hence, by using \eqref{E_jn+1n} recursively, we have the error estimate for the FKE between two observations from $t_{j-1}$ to $t_j$, 
\begin{align}
E_j^{\frac{\Delta T}{\tau}} \leq e^{\Delta TdL_f} E_j^0 + \frac{C}{dL_f}(\tau+h^2),
\label{E-DeltaT-tau}
\end{align}
which gives the estimate \eqref{FDtopathwiseDMZconvergence1}, i.e. the error estimate in the time interval $[t_{j-1},t_j]$.  Notice that the initial data for the FD scheme \eqref{FDscheme_for_FKE} is obtained by an exponential transform, which assimilates the new observation data into the predicted solution at $t=t_{j-1}$
computed on the time interval $[t_{j-2},t_{j-1}]$. Thus, the term $E^0_{j}$ satisfies 
\begin{align}
E^0_{j} \leq \exp(\textbf{h}^T\big(\textbf{y}_{t_{j-1}}-\textbf{y}_{t_{j-2}})\big)E^{\frac{\Delta T}{\tau}} _{j-1} \leq e^{c_2}E^{\frac{\Delta T}{\tau}} _{j-1}.
\label{E0}
\end{align}
Combining \eqref{E-DeltaT-tau} and \eqref{E0},  we obtain 
\begin{align}
	E^{\frac{\Delta T}{\tau}} _{j}   \leq e^{c_2+\Delta TdL_f}E^{\frac{\Delta T}{\tau}} _{j-1} +\frac{C}{dL_f}(\tau+h^2).
\end{align}
Recursively using the above estimate and the condition $E_1^0=0$, we have 
\begin{align}
E^{\frac{\Delta T}{\tau}} _{N_t} & \leq e^{(c_2+\Delta TdL_f)(N_t-1)}\Big(E^{\frac{\Delta T}{\tau}} _{1} +\frac{C(\tau+h^2)}{dL_f(e^{c_2+\Delta TdL_f}-1)}\Big),\nonumber \\
& \leq e^{(c_2+\Delta TdL_f)(N_t-1)}\Big(e^{\Delta TdL_f} E_1^0 + \frac{C(\tau+h^2)}{dL_f} + \frac{C(\tau+h^2)}{dL_f(e^{c_2+\Delta TdL_f}-1)}\Big), \nonumber \\
&  \leq\frac{Ce^{c_2N_t+TdL_f}}{dL_f(e^{c_2+\Delta TdL_f}-1)}(\tau+h^2),
\end{align}	
which gives the estimate \eqref{FDtopathwiseDMZconvergence2}. 
\end{proof}

\subsection{Convergence analysis for the QTT method}
\noindent
Finally, we analyze the error between the solutions obtained by using the QTT method and the FD method. Let $\epsilon_1$ denote a given precision in the construction of QTT-format and TT-rounding, and $\epsilon_2$ denote the error (in the sense of Frobenius norm) of operator $(\tau \textbf{A}+ \textbf{I})^{\frac{\Delta T}{\tau}}$ between the FD matrix and QTT-format approximation matrix, respectively. 

We first analyze the convergence of the QTT solution to the FD solution. Note that the QTT method gives approximate solutions only at time $t_j$, $j=0,1,\cdots,N_t$. Hence in the following analysis, 
let $U_{\textbf{l},j}^{0}$ denote the FD solutions at time $t_{j-1}$, $j=1,...,N_t$ after the exponential transformation and let $U_{\textbf{l},j}^{\frac{\Delta T}{\tau}}$ denote the FD solutions at time 
$t_j$, $j=1,...,N_t$ before the exponential transformation, respectively.
\begin{lemma}\label{lemma_of_TTconvergence}
Let $\reallywidehat{U}_{\textbf{l},j}^{0}$ denote the QTT solutions at time $t_{j-1}$, $j=1,...,N_t$ after the exponential transformation and let $\reallywidehat{U}_{\textbf{l},j}^{\frac{\Delta T}{\tau}}$ denote the QTT solutions at time $t_j$, $j=1,...,N_t$ before the exponential transformation, respectively. 	
We have the error estimate
	\begin{align} \label{QTTtoFDconvergence1}
	\big|\big|\reallywidehat{U}_{\textbf{l},j}^{\frac{\Delta T}{\tau}}-U_{\textbf{l},j}^{\frac{\Delta T}{\tau}}\big|\big|_2 \leq  (1+\epsilon_2)c_6\big|\big| \reallywidehat{U}_{\textbf{l},j}^{0} -U_{\textbf{l},j}^{0}\big|\big|_2 +\epsilon_2c_6 \big|\big|U_{\textbf{l},j}^{0}\big|\big|_2,
	\end{align}
	%on $[t_{j-1},t_j]$, 
	where $c_6=\big|\big|(\tau \textbf{A}+ \textbf{I})^{\frac{\Delta T}{\tau}}\big|\big|_2$. Finally, we have the error estimate at the final time
	\begin{small}
	\begin{align}\label{QTTtoFDconvergence2}
	&\big|\big|\reallywidehat{U}_{\textbf{l},N_t}^{\frac{\Delta T}{\tau}}-U_{\textbf{l},N_t}^{\frac{\Delta T}{\tau}}\big|\big|_2 \leq c_7^{N_t}\Big(\epsilon_1 \big|\big| U_{\textbf{l},0}^{\frac{\Delta T}{\tau}}\big|\big|_2 \Big)+ \sum_{j=1}^{N_t}c_7^{N_t-j}\Big((1+\epsilon_2)\epsilon_1c_6e^{c_2}\big|\big|U_{\textbf{l},j-1}^{\frac{\Delta T}{\tau}}\big|\big|_2+\epsilon_2c_6 \big|\big|U_{\textbf{l},j}^{0}\big|\big|_2\Big),
	\end{align}
	\end{small}
	where $c_7=(1+\epsilon_2)(1+\epsilon_1)c_6e^{c_2}$. 
\end{lemma}
\begin{proof}
Notice that the online procedure, i.e., the Algorithm \ref{alg:OnlineStageTT}, is divided into two main parts of computations. The first part is assimilating the observation data into the predicted solution with a TT-rounding procedure afterward. By using the triangular inequality and stability of the FD scheme, we easily obtain 
\begin{align}
\big|\big|\reallywidehat{U}_{\textbf{l},j+1}^{0}-U_{\textbf{l},j+1}^{0}\big|\big|_{2} \leq & \big|\big|\reallywidehat{\exp[\textbf{h}^T_{j}(\textbf{y}_{t_{j}}-\textbf{y}_{t_{j-1}})]}\reallywidehat{U}_{\textbf{l},j}^{\frac{\Delta T}{\tau}} - \exp[\textbf{h}^T_{j}(\textbf{y}_{t_{j}}-\textbf{y}_{t_{j-1}})]U_{\textbf{l},j}^{\frac{\Delta T}{\tau}} \big|\big|_2, \nonumber \\
\leq& \big|\big|\exp[\textbf{h}^T_{j}(\textbf{y}_{t_{j}}-\textbf{y}_{t_{j-1}})] \big|\big|_2 \Big(\big|\big| \reallywidehat{U}_{\textbf{l},j}^{\frac{\Delta T}{\tau}} -U_{\textbf{l},j}^{\frac{\Delta T}{\tau}}\big|\big|_2  \Big)  \nonumber \\
&+\Big(\big|\big|\exp[\textbf{h}^T_{j}(\textbf{y}_{t_{j}}-\textbf{y}_{t_{j-1}})]- \reallywidehat{\exp[\textbf{h}^T_{j}(\textbf{y}_{t_{j}}-\textbf{y}_{t_{j-1}})]} \big|\big|_2\Big)\big|\big|\reallywidehat{U}_{\textbf{l},j}^{\frac{\Delta T}{\tau}}\big|\big|_2, \nonumber\\
\leq &e^{c_2}\big|\big| \reallywidehat{U}_{\textbf{l},j}^{\frac{\Delta T}{\tau}} -U_{\textbf{l},j}^{\frac{\Delta T}{\tau}}\big|\big|_2 + \epsilon_1e^{c_2}(\big|\big|\reallywidehat{U}_{\textbf{l},j}^{\frac{\Delta T}{\tau}}-U_{\textbf{l},j}^{\frac{\Delta T}{\tau}}\big|\big|_2 + \big|\big|U_{\textbf{l},j}^{\frac{\Delta T}{\tau}}\big|\big|_2), \nonumber\\
\leq &(1+\epsilon_1)e^{c_2}\big|\big| \reallywidehat{U}_{\textbf{l},j}^{\frac{\Delta T}{\tau}} -U_{\textbf{l},j}^{\frac{\Delta T}{\tau}}\big|\big|_2 + \epsilon_1e^{c_2}\big|\big|U_{\textbf{l},j}^{\frac{\Delta T}{\tau}}\big|\big|_2, 
\label{estimate-part1} 
\end{align}
where $\reallywidehat{\exp[\textbf{h}^T_{j}(\textbf{y}_{t_{j}}-\textbf{y}_{t_{j-1}})]}$ denotes the QTT-format of the vector $\exp[\textbf{h}^T_{j}(\textbf{y}_{t_{j}}-\textbf{y}_{t_{j-1}})]$.

The second part of computation is a matrix-vector multiplication in the QTT-format with a TT-rounding procedure. By using the triangular inequality and the fact $\big|\big|\textbf{M}\big|\big|_2\leq\big|\big|\textbf{M}\big|\big|_F$ for any matrix $\textbf{M}$, we easily obtain
\begin{align} 
&\big|\big|\reallywidehat{U}_{\textbf{l},j+1}^{\frac{\Delta T}{\tau}}-U_{\textbf{l},j+1}^{\frac{\Delta T}{\tau}}\big|\big|_2 \leq  \big|\big|\reallywidehat{(\tau \textbf{A}+ \textbf{I})^{\frac{\Delta T}{\tau}}}\reallywidehat{U}_{\textbf{l},j+1}^{0} - (\tau \textbf{A}+ \textbf{I})^{\frac{\Delta T}{\tau}}U_{\textbf{l},j+1}^{0}\big|\big|, \nonumber  \\
\leq & \big|\big|(\tau \textbf{A}+ \textbf{I})^{\frac{\Delta T}{\tau}}\big|\big|_2 \Big(\big|\big| \reallywidehat{U}_{\textbf{l},j+1}^{0} -U_{\textbf{l},j+1}^{0}\big|\big|_2 \Big) + \Big(\big|\big|(\tau \textbf{A}+ \textbf{I})^{\frac{\Delta T}{\tau}} - \reallywidehat{(\tau \textbf{A}+ \textbf{I})^{\frac{\Delta T}{\tau}}}\big|\big|_2 \Big) \big|\big|\reallywidehat{U}_{\textbf{l},j+1}^{0}\big|\big|_2, \nonumber\\
\leq & c_6\big|\big| \reallywidehat{U}_{\textbf{l},j+1}^{0} -U_{\textbf{l},j+1}^{0}\big|\big|_2 + \epsilon_2c_6\Big(\big|\big| \reallywidehat{U}_{\textbf{l},j+1}^{0} -U_{\textbf{l},j+1}^{0}\big|\big|_2 + \big|\big|U_{\textbf{l},j+1}^{0}\big|\big|_2\Big), \nonumber\\
\leq & (1+\epsilon_2)c_6\big|\big| \reallywidehat{U}_{\textbf{l},j+1}^{0} -U_{\textbf{l},j+1}^{0}\big|\big|_2 +\epsilon_2c_6 \big|\big|U_{\textbf{l},j+1}^{0}\big|\big|_2, 
\label{estimate-part2}
\end{align}
where $\reallywidehat{(\tau \textbf{A}+ \textbf{I})^{\frac{\Delta T}{\tau}}}$ is the QTT-format of the operator matrix $(\tau \textbf{A}+ \textbf{I})^{\frac{\Delta T}{\tau}}$ and we have denoted $c_6=\big|\big|(\tau \textbf{A}+ \textbf{I})^{\frac{\Delta T}{\tau}}\big|\big|_2$. The estimate \eqref{QTTtoFDconvergence1} is proved. 

Combining the above two estimates \eqref{estimate-part1} and \eqref{estimate-part2}, we get 
\begin{align} 
\big|\big|\reallywidehat{U}_{\textbf{l},j+1}^{\frac{\Delta T}{\tau}}-U_{\textbf{l},j+1}^{\frac{\Delta T}{\tau}}\big|\big|_2 \leq & (1+\epsilon_2)(1+\epsilon_1)c_6e^{c_2}\big|\big|\reallywidehat{U}_{\textbf{l},j}^{\frac{\Delta T}{\tau}}-U_{\textbf{l},j}^{\frac{\Delta T}{\tau}}\big|\big|_2 \nonumber \\ 
&+(1+\epsilon_2)\epsilon_1c_6e^{c_2}\big|\big|U_{\textbf{l},j}^{\frac{\Delta T}{\tau}}\big|\big|_2+\epsilon_2c_6 \big|\big|U_{\textbf{l},j+1}^{0}\big|\big|_2.
\label{estimate-combine}
\end{align}
We denote $c_7=(1+\epsilon_2)(1+\epsilon_1)c_6e^{c_2}$ for notational simplicity. Recursively using the above estimate \eqref{estimate-combine}, we obtain 
\begin{align}
\big|\big|\reallywidehat{U}_{\textbf{l},N_t}^{\frac{\Delta T}{\tau}}-U_{\textbf{l},N_t}^{\frac{\Delta T}{\tau}}\big|\big|_2 &\leq c_7^{N_t}\Big( \big|\big|\reallywidehat{U}_{\textbf{l},0}^{\frac{\Delta T}{\tau}}-U_{\textbf{l},0}^{\frac{\Delta T}{\tau}}\big|\big|_2 \Big)
+\sum_{j=1}^{N_t}c_7^{N_t-j}\Big((1+\epsilon_2)\epsilon_1c_6e^{c_2}\big|\big|U_{\textbf{l},j-1}^{\frac{\Delta T}{\tau}}\big|\big|_2+\epsilon_2c_6 \big|\big|U_{\textbf{l},j}^{0}\big|\big|_2\Big),\nonumber\\
&\leq c_7^{N_t}\Big(\epsilon_1 \big|\big| U_{\textbf{l},0}^{\frac{\Delta T}{\tau}}\big|\big|_2 \Big) +\sum_{j=1}^{N_t}c_7^{N_t-j}\Big((1+\epsilon_2)\epsilon_1c_6e^{c_2}\big|\big|U_{\textbf{l},j-1}^{\frac{\Delta T}{\tau}}\big|\big|_2+\epsilon_2c_6 \big|\big|U_{\textbf{l},j}^{0}\big|\big|_2\Big), 
%&\longrightarrow 0  \text{~~~as~ }\epsilon_1,\epsilon_2 \rightarrow 0,\nonumber
\end{align} 
which completes the proof.
\end{proof}
\begin{remark}
The estimate \eqref{QTTtoFDconvergence2} reveals the dependence of the error of the QTT solution on different parameters. Since $\big|\big|U_{\cdot,j}^{\frac{\Delta T}{\tau}}\big|\big|_2$, $j=0,....,N_t$ are bounded, 
given parameters $c_6$, $c_7$, and $N_t$, one can choose $\epsilon_1$ and $\epsilon_2$ accordingly so that the 
error $\big|\big|\reallywidehat{U}_{\textbf{l},N_t}^{\frac{\Delta T}{\tau}}-U_{\textbf{l},N_t}^{\frac{\Delta T}{\tau}}\big|\big|_2$ is small. 
\end{remark}

As an immediate result of Lemma \ref{lemma_of_FDconvergence} and Lemma \ref{lemma_of_TTconvergence}, we can estimate the error between the solution of the QTT method and the exact solution for the FKE \eqref{FKEinYauProposition} on each time interval $[t_{j-1},t_j]$. We provide such an estimate in the following lemma, of which the proof is a simple application of the triangular inequality. 
\begin{lemma}\label{lemma_of_TTconvergenceExact}
Let $\reallywidehat{U}_{\textbf{l},j}^{\frac{\Delta T}{\tau}}$ denote the solution of the QTT method and $\widetilde{u}_j(\textbf{x}_\textbf{l},t_{j})$ denote the exact solution for FKE \eqref{FKEinYauProposition} 
at time $t_j$, $j=1,...,N_t$, respectively. We have the error estimate
\begin{align} \label{TTconvergenceExact}
\big|\big|\reallywidehat{U}_{\textbf{l},j}^{\frac{\Delta T}{\tau}}-\widetilde{u}_j(\textbf{x}_\textbf{l},t_{j})\big|\big|_\infty
\leq & (\epsilon_1+\epsilon_1\epsilon_2+\epsilon_2)c_6 \big|\big|U_{\textbf{l},j}^{0}\big|\big|_2 +\frac{C(\tau+h^2)}{dL_f}.  
% &\longrightarrow 0  \text{~~~as~ }\epsilon_1,\epsilon_2,\tau,h \rightarrow 0,\nonumber
\end{align}
\end{lemma}
\begin{proof}
From the error estimate \eqref{FDtopathwiseDMZconvergence1} in Lemma \ref{lemma_of_FDconvergence} and 
the estimate \eqref{QTTtoFDconvergence1} in Lemma \ref{lemma_of_TTconvergence}, and the infinity norm of any vector is bounded by its 2-norm, we have
	\begin{align} 
	&\big|\big|\reallywidehat{U}_{\textbf{l},j}^{\frac{\Delta T}{\tau}}-\widetilde{u}_j(\textbf{x}_\textbf{l},t_{j})\big|\big|_\infty \leq  \big|\big|\reallywidehat{U}_{\textbf{l},j}^{\frac{\Delta T}{\tau}}-U_{\textbf{l},j}^{\frac{\Delta T}{\tau}}\big|\big|_\infty+||U_{\textbf{l},j}^{\frac{\Delta T}{\tau}} - \widetilde{u}_j(\textbf{x}_\textbf{l},t_{j})||_\infty\nonumber\\
	\leq & (1+\epsilon_2)c_6\big|\big| \reallywidehat{U}_{\textbf{l},j}^{0} -U_{\textbf{l},j}^{0}\big|\big|_2 +\epsilon_2c_6 \big|\big|U_{\textbf{l},j}^{0}\big|\big|_2
	+ e^{\Delta TdL_f}||U_{\textbf{l},j}^{0} - \widetilde{u}_j(\textbf{x}_\textbf{l},t_{j-1})||_\infty+\frac{C}{dL_f}(\tau+h^2).
	\end{align}
	Since at each initial time $t_{j-1}$ the FD method uses the exact initial data $\widetilde{u}_j(\textbf{x}_\textbf{l},t_{j-1})$ to compute, so we have $||U_{\textbf{l},j}^{0} - \widetilde{u}_j(\textbf{x}_\textbf{l},t_{j-1})||_\infty=0$. In addition, we have $\big|\big| \reallywidehat{U}_{\textbf{l},j}^{0} -U_{\textbf{l},j}^{0}\big|\big|_2\leq \epsilon_1\big|\big|U_{\textbf{l},j}^{0}\big|\big|_2$. Thus, the error estimate in \eqref{TTconvergenceExact} 
	is proved. 
%	Assuming that there is no error introduced at initial time $t_{j-1}$, i.e. $||U_{\textbf{l},j}^{0} - \widetilde{u}_j(\textbf{x}_\textbf{l},t_{j-1})||_\infty=0$ and , we obtain the convergence result \eqref{TTconvergenceExact}.
%	\begin{align} 
%	\big|\big|u_{j}^{QTT,-}-u_j(\cdot,t_j)\big|\big|_\infty
%	\leq & (1+\epsilon_2)c_6\epsilon_1\big|\big|u_{j-1}^{FD}\big|\big|_2 +\epsilon_2c_6 \big|\big|u_{j-1}^{FD}\big|\big|_2 +\frac{C(\tau+h^2)}{dL_f} \\
%	&\longrightarrow 0  \text{~~~as~ }\epsilon_1,\epsilon_2,\tau,h \rightarrow 0,\nonumber
%	\end{align}
%	which gives the convergence result \eqref{TTconvergenceExact}.
\end{proof}

Now we are in a position to present the main result. Recall that $u(\textbf{x},t)$ denote the solution of pathwise robust DMZ equation \eqref{TheoremCeq1} on $\Omega\times[0,T]$, $u_{N_t}(\textbf{x},t)$ denote the solution of frozen time equation \eqref{TheoremCeq2} on $\Omega\times[t_{N_t-1},t_{N_t}]$. 
Let $\mathcal{I}$ denote an interpolation operator, which can be a polynomial interpolation or spline interpolation, 
and let $\reallywidehat{u}(\textbf{x},t_j)=\mathcal{I} \reallywidehat{U}_{\textbf{l},j}^{\frac{\Delta T}{\tau}}$ 
denote the recovered function on $\Omega$ based on the QTT solution.  
%$\reallywidehat{U}_{\textbf{l},j}^{\frac{\Delta T}{\tau}}$ at time $t_j$, we can  obtain $\reallywidehat{u}(\textbf{x},t_j)=\mathcal{I} \reallywidehat{U}_{\textbf{l},j}^{\frac{\Delta T}{\tau}}$. 
%Similarly, let $\mathcal{R}$ denote a restriction operator, which gives us the function values on grid point.  
\begin{theorem}\label{QTTFianlConverenceResult}
The QTT solution converges to the solution of the pathwise robust DMZ equation on a bounded domain in the sense of $L^1$ norm. Specifically, we have the error estimate as follows, 
\begin{align}\label{finaltheoremFINAL}
\big|\big|&u(\textbf{x},T)-\exp(-\textbf{h}(\textbf{x})^T\textbf{y}_{t_{N_t-1}})\reallywidehat{u}(\textbf{x},t_{N_t})\big|\big|_{L^1}\leq \frac{2c_5}{\alpha+1}\frac{T^{\alpha+1}e^{c_3T}}{{N_t}^\alpha}\nonumber\\
&+ Ce^{c_2}\Big( \frac{Ce^{c_2N_t+TdL_f}}{dL_f(e^{c_2+\Delta TdL_f}-1)}(\tau+h^2) + h^2||\Delta \widetilde{u}_{N_t}(\textbf{x},T) ||_{L^2}\Big)\nonumber\\
&+Ce^{c_2}c_7^{N_t}\Big(\epsilon_1 \big|\big| U_{\textbf{l},0}^{\frac{\Delta T}{\tau}}\big|\big|_2 \Big)+Ce^{c_2}\sum_{j=1}^{N_t}c_7^{N_t-j}\Big((1+\epsilon_2)\epsilon_1c_6e^{c_2}\big|\big|U_{\textbf{l},j-1}^{\frac{\Delta T}{\tau}}\big|\big|_2+\epsilon_2c_6 \big|\big|U_{\textbf{l},j}^{0}\big|\big|_2\Big).
\end{align}
\end{theorem}
\begin{proof}
By the triangle inequality, we split the error into three main parts
\begin{align}\label{Triinequality-SplitError}
\big|\big|&u(\textbf{x},T)-\exp(-\textbf{h}(\textbf{x})^T\textbf{y}_{t_{N_t-1}})\reallywidehat{u}(\textbf{x},t_{N_t})\big|\big|_{L^1} \nonumber \\
\leq & \big|\big|u(\textbf{x},T)-u_{N_t}(\textbf{x},T)\big|\big|_{L^1} + \big|\big|\exp(-\textbf{h}(\textbf{x})^T\textbf{y}_{t_{N_t-1}})\widetilde{u}_{N_t}(\textbf{x},T)-\exp(-\textbf{h}(\textbf{x})^T\textbf{y}_{t_{N_t-1}})\mathcal{I}U_{\textbf{l},N_t}^{\frac{\Delta T}{\tau}}\big|\big|_{L^1}
\nonumber \\ 
&+\big|\big|\exp(-\textbf{h}(\textbf{x})^T\textbf{y}_{t_{N_t-1}})\mathcal{I}U_{\textbf{l},N_t}^{\frac{\Delta T}{\tau}}-\exp(-\textbf{h}(\textbf{x})^T\textbf{y}_{t_{N_t-1}})\mathcal{I}\reallywidehat{U}_{\textbf{l},N_t}^{\frac{\Delta T}{\tau}}\big|\big|_{L^1}, \nonumber \\
:=& E_1 + E_2 + E_3, 
\end{align}
where we have used the condition $u_{N_t}(\textbf{x},T) = \exp(-\textbf{h}(\textbf{x})^T\textbf{y}_{t_{N_t-1}})\widetilde{u}_{N_t}(\textbf{x},T)$ (see Eq.\eqref{u-equal-exputitle}). 	In what follows, we shall estimate these three error terms separately. 

From Prop.\ref{YauYau:theoremC}, we know the error term $E_1$ satisfies   
\begin{align}\label{E1}
E_1=||u(\textbf{x},T)-u_{N_t}(\textbf{x},T)||_{L^1} \leq \frac{2c_5}{\alpha+1}\frac{T^{\alpha+1}e^{c_3T}}{{N_t}^\alpha}.
\end{align}
 
Then, we have the estimate for the error term $E_2$ as follows, 
\begin{align}\label{E2}
&E_2=||\exp(-\textbf{h}(\textbf{x})^T\textbf{y}_{t_{N_t-1}})\widetilde{u}_{N_t}(\textbf{x},T)-\exp(-\textbf{h}(\textbf{x})^T\textbf{y}_{t_{N_t-1}})\mathcal{I}U_{\textbf{l},N_t}^{\frac{\Delta T}{\tau}} ||_{L^1}\nonumber\\
\leq & e^{c_2}||\widetilde{u}_{N_t}(\textbf{x},T) - \mathcal{I}U_{\textbf{l},N_t}^{\frac{\Delta T}{\tau}} ||_{L^1}\nonumber\\
\leq & e^{c_2}\Big(||\mathcal{I}U_{\textbf{l},N_t}^{\frac{\Delta T}{\tau}} -\mathcal{I}\widetilde{u}_{N_t}(\textbf{x}_\textbf{l},T)||_{L^1} + ||\mathcal{I}\widetilde{u}_{N_t}(\textbf{x}_\textbf{l},T)-\widetilde{u}_{N_t}(\textbf{x},T) ||_{L^1}\Big)\nonumber\\
%\leq & e^{c_2}C\Big(||\mathcal{I}U_{\textbf{l},N_t}^{\frac{\Delta T}{\tau}} -\mathcal{I}\widetilde{u}_{N_t}(\textbf{x}_\textbf{l},T)||_{L^\infty} + ||\mathcal{I}\widetilde{u}_{N_t}(\textbf{x}_\textbf{l},T)-\widetilde{u}_{N_t}(\textbf{x},T) ||_{L^2}\Big)\nonumber\\
\leq & Ce^{c_2}\Big(||U_{\textbf{l},N_t}^{\frac{\Delta T}{\tau}} -\widetilde{u}_{N_t}(\textbf{x}_\textbf{l},T)||_{\infty} + h^2||\Delta\widetilde{u}_{N_t}(\textbf{x},T) ||_{L^2}\Big)\nonumber\\
\leq & Ce^{c_2}\Big( \frac{Ce^{c_2N_t+TdL_f}}{dL_f(e^{c_2+\Delta TdL_f}-1)}(\tau+h^2) + h^2||\Delta \widetilde{u}_{N_t}(\textbf{x},T) ||_{L^2}\Big)
\end{align}
where $C$ is a generic constant depending on $\Omega$. Here we have used the facts that the norm of the operator 
$\mathcal{I}$ is bounded and the interpolation error has second-order convergence rate,  since the function $\widetilde{u}_{N_t}(\textbf{x},T)$ is smooth. In addition, 
the estimate \eqref{FDtopathwiseDMZconvergence2} in Lemma \ref{lemma_of_FDconvergence} is used. 

Finally, we estimate the error term $E_3$ and get 
\begin{align}\label{E3}
&\big|\big|\exp(-\textbf{h}(\textbf{x})^T\textbf{y}_{t_{N_t-1}})\mathcal{I}\reallywidehat{U}_{\textbf{l},N_t}^{\frac{\Delta T}{\tau}}-\exp(-\textbf{h}(\textbf{x})^T\textbf{y}_{t_{N_t-1}})\mathcal{I}U_{\textbf{l},N_t}^{\frac{\Delta T}{\tau}}\big|\big|_{L^1}\nonumber\\
\leq & e^{c_2}\big|\big|\mathcal{I}\reallywidehat{U}_{\textbf{l},N_t}^{\frac{\Delta T}{\tau}}-\mathcal{I}U_{\textbf{l},N_t}^{\frac{\Delta T}{\tau}}\big|\big|_{L^1} 
\leq  Ce^{c_2}\big|\big|\reallywidehat{U}_{\textbf{l},N_t}^{\frac{\Delta T}{\tau}}-U_{\textbf{l},N_t}^{\frac{\Delta T}{\tau}}\big|\big|_{2}\nonumber\\
\leq & Ce^{c_2}c_7^{N_t}\Big(\epsilon_1 \big|\big| U_{\textbf{l},0}^{\frac{\Delta T}{\tau}}\big|\big|_2 \Big)+ 
Ce^{c_2}\sum_{j=1}^{N_t}c_7^{N_t-j}\Big((1+\epsilon_2)\epsilon_1c_6e^{c_2}\big|\big|U_{\textbf{l},j-1}^{\frac{\Delta T}{\tau}}\big|\big|_2+\epsilon_2c_6 \big|\big|U_{\textbf{l},j}^{0}\big|\big|_2\Big)
\end{align}
where $C$ is a constant depending on $\Omega$ and the estimate \eqref{QTTtoFDconvergence2} in Lemma \ref{lemma_of_TTconvergence} is used.  Combining above formulas \eqref{E1}\eqref{E2}\eqref{E3}, 
we prove the statement in the theorem \ref{QTTFianlConverenceResult}. 
\end{proof}
\begin{remark}
The accuracy of the QTT method in computing the the solution of pathwise robust DMZ equation is
controlled by three components of approximation errors, i.e. $E_1$, $E_2$, and $E_3$. 
In practice, we can choose $N_t$,  $\tau$, $h$, $\epsilon_1$ and $\epsilon_2$ accordingly so that 
the error $\big|\big|u(\textbf{x},T)-\exp(-\textbf{h}(\textbf{x})^T\textbf{y}_{t_{N_t-1}})\reallywidehat{u}(\textbf{x},t_{N_t})\big|\big|_{L^1}$ is small.
\end{remark}

%When $\epsilon_1,\epsilon_2\rightarrow 0$, $\tau,h\rightarrow 0$, $N_t\rightarrow\infty$ sequentially, we have $\big|\big|\exp(-\textbf{h}(\textbf{x})^T\textbf{y}_{t_{N_t-1}})\mathcal{I}u_{N_t}^{QTT,-}-u^{rb}(\textbf{x},T)\big|\big|_{L^1}\longrightarrow 0$. 
%\textcolor{red}{??? Could you please try you best to summarize your result, i.e., the convergence between QTT solution and exact solution (on bounded domain) here. Please (1) use the correct notations; (2) give the clear and right convergence rate??? Try you best to write a result here and we shall discuss soon.}

\section{Numerical results}\label{sec:Numerical_example}
\noindent 
In this section, we are interested in investigating the accuracy and efficiency of our method in solving the NLF problems. We shall carry out numerical experiments on two 3D NLF problems. The definitions of these two NLF problems are given as follows, 

%\textcolor{red}{
%\begin{itemize}
%\item two examples: almost linear, cubic
%\item Table: online computation ( prescibed epsilon)(2d); column: spatial discretization points, row: CPU time, QTT-erank of $u(x,T)$.
%\item Number: MSE error. 1st between QTT and Finite difference, 2nd between QTT and real state.
%\end{itemize}} 
 
\underline{Example 1: An almost linear problem} This problem is modeled by a SDE in the Ito form as follows,
\begin{align}\label{NonlinearSignalModel_Experiment1}
\begin{cases}
dx_{1} = -0.3x_1 + dv_1, \\
dx_{2} = -0.3x_2 + dv_2, \\
dx_{3} = -0.3x_3 + dv_3, \\
dy_{1} = (x_2+\sin(x_1))dt +   dw_1, \\
dy_{2} = (x_3+\sin(x_2))dt +   dw_2, \\
dy_{3} = (x_1+\sin(x_3))dt +   dw_3. \\
\end{cases}
\end{align}
where $E[d\vec{v}_td\vec{v}_t^{T}]= 1.5I_3dt$ with $\vec{v}=[v_1,v_2,v_3]^T$, $E[d\vec{w}_td\vec{w}_t^T]=I_3 dt$ with $\vec{w}=[w_1,w_2,w_3]^T$, and $I_3$ is the identity matrix of size $3\times3$. The noise are independent Brownian motions. The initial state is $\vec{x}(0)=[x_1(0),x_2(0),x_3(0)]^T=[0,0,0]^T$ with $\vec{x}(t)=[x_1(t),x_2(t),x_3(t)]^T$. 

\underline{Example 2: A cubic sensor problem} This problem is modeled by a SDE in the Ito form as follows,
\begin{align}\label{NonlinearSignalModel_Experiment2}
\begin{cases}
dx_{1} = (-0.6x_1- 0.1x_2)dt + dv_1, \\
dx_{2} = (-0.5x_2+0.1x_3)dt + dv_2, \\
dx_{3} = (-0.6x_3+0.1x_1)dt + dv_3, \\
dy_{1} = x_2^3dt +   dw_1, \\
dy_{2} = x_3^3dt +   dw_2, \\
dy_{3} = x_1^3dt +   dw_3, \\
\end{cases}
\end{align}
where $E[d\vec{v}_td\vec{v}_t^{T}]= 1.5I_3dt$ with $\vec{v}=[v_1,v_2,v_3]^T$, $E[d\vec{w}_td\vec{w}_t^T]=I_3dt$ with $\vec{w}=[w_1,w_2,w_3]^T$, $I_3$ is the identity matrix of size $3\times3$. The initial state is $\vec{x}(0)=[x_1(0),x_2(0),x_3(0)]^T=[0,0,0]^T$ with $\vec{x}(t)=[x_1(t),x_2(t),x_3(t)]^T$. 

The total experimental time is $T=20s$ for both examples. The cubic sensor problem has higher nonlinearity than the almost linear one. Thus, it is more difficult. 
%and make the extended Kalman filter method perform badly.

\subsection{QTT-ranks in the spatial discretization}
\noindent
We consider the spatial discretization of operators on the 3D domain in the QTT-format. We have shown that the exact QTT-decomposition of the discretized Laplace operator via standard FD scheme and first partial derivative operator via central difference scheme have low QTT-ranks. Hence, we mainly compute 
the QTT-ranks of the discretization of velocity field $\vec{f}(\vec{x})=[f_1(\vec{x}),f_2(\vec{x}),f_3(\vec{x})]^T$ and function $\vec{h}^T\textbf{S}^{-1}\vec{h}$. 

Let us define the effective QTT-rank of a QTT-format as
\begin{align} r_{eff}:=\frac{\sqrt{(r_0n_1+r_dn_d)^2+4(\sum_{k=2}^{d-1}n_k)\sum_{k=1}^{d}r_{k-1}n_kr_k}-(r_0n_1+r_dn_d)}{2\sum_{k=2}^{d-1}n_k}, 
\label{effectiveQTTranks}
\end{align} 
where $n_k,r_k$ representing mode sizes and QTT-ranks. Notice that $n_k=2, k=1,\cdots,d,$ in the QTT-format case.
In Table \ref{tab:erank1} and Table \ref{tab:erank2}, we show the effective QTT-rank for the Example 1 and Example 2, respectively. We observe a very slow growth in the effective QTT-rank with respect to the degree of freedom in the spatial discretization.
\begin{table}[htbp]
	\begin{center}
		\begin{tabular}{|c|cccc|}
			\hline
		        $N$ on each direction &   $f_1(\vec{x})$ & $f_2(\vec{x})$ & $f_3(\vec{x})$ & $\vec{h}^T\textbf{S}^{-1}\vec{h}$\\
			\hline 
			 $2^4$  &  1.32 & 1.32 & 1.32 &  4.77 \\
			 \hline
			 $2^5$   &  1.34 & 1.34 & 1.34 &   5.41\\	
			 \hline
			 $2^6$ & 2.20 & 1.35 & 2.04 &  5.82 \\
			 \hline
			 $2^7$ & 2.23 & 2.34 & 2.29 &  6.08\\
			 \hline
			 $2^8$  & 2.40 & 2.35 & 2.30& 6.27\\
			\hline
		\end{tabular}%
	\end{center}
	\caption{Effective QTT-rank of the discretized functions on the spatial grid with a given precision 
		$1\times 10^{-12}$ in the almost linear problem.}
	\label{tab:erank1}
\end{table}%

\begin{table}[htbp]
	\begin{center}
		\begin{tabular}{|c|cccc|}
			\hline
		        $N$ on each direction &   $f_1(\vec{x})$ & $f_2(\vec{x})$ & $f_3(\vec{x})$ & $\vec{h}^T\textbf{S}^{-1}\vec{h}$\\
			\hline 
			$2^4$ &1.69 & 1.69 & 2.00 &3.03  \\	
			\hline
			 $2^5$ &1.70 & 1.70 & 2.00 &3.53   \\
			 \hline	
			 $2^6$ & 1.71 & 1.71 & 2.00   &4.27\\
			 \hline
			 $2^7$ & 2.65 & 2.65 & 2.93  &4.80\\
			 \hline
			$2^8$ & 2.66 & 2.63 & 2.94  &5.15\\
			\hline
		\end{tabular}%
	\end{center}
	\caption{Effective QTT-rank of the discretized functions on the spatial grid with a given precision 
		$1\times 10^{-12}$ in the cubic sensor problem.}
	\label{tab:erank2}
\end{table}%

In Table \ref{tab:erank3} and Table \ref{tab:erank4}, we show the effective QTT-ranks of assembled tensors $(\tau \textbf{A}+\textbf{I})$ and $(\tau \textbf{A}+\textbf{I})^{\frac{\Delta T}{\tau}}$ that are pre-computed in the offline procedure; see Eq.\eqref{AddingOperators_FKE} and Eq.\eqref{AssemblingOperators_FKE}. 
In our experiments, we set $\Delta T = 0.05$ and $\frac{\Delta T}{\tau}=100$ in Example 1 ($\frac{\Delta T}{\tau}=200$ in Example 2). Recall the $\Delta T$ is the time between two observations and $\tau$ is the time step in discretizing the FKE \eqref{PathwiseRobustDMZ_KFE}. The time step $\tau$ in Example 1 and Example 2 is chosen in such a way that the Courant-Friedrichs-Lewy (CFL) stability condition is satisfied \cite{morton2005numerical}. 
Note that the requirement of a small time step $\tau$ makes the FD method expensive for solving high-dimensional and/or nonlinear problems. While in our QTT method, the trouble caused by small $\tau$ is avoided  
since we can compute $(\tau \textbf{A}+\textbf{I})^{\frac{\Delta T}{\tau}}$ and approximated it using the 
QTT method in the offline procedure. Moreover, we can prove that the accuracy of $(\tau \textbf{A}+\textbf{I})^{\frac{\Delta T}{\tau}}$ is bounded by $\frac{\Delta T}{\tau}\epsilon$, if the TT-rounding precision $\epsilon$ is given. 

\begin{table}[htbp]
	\begin{center}
		\begin{tabular}{|c|c|c|}
%			\hline
%			 & \multicolumn{2}{c|}{Example 1}& \multicolumn{2}{c|}{Example 2}\\
			\hline
			spatial $N$ on each direction  &  Example 1 & Example 2 \\
%		         spatial $N$ on each direction    &   $(\tau A+ I)$ & $(\tau A+ I)^{\frac{\Delta T}{\tau}}$ & $(\tau A+ I)$ & $(\tau A+ I)^{\frac{\Delta T}{\tau}}$\\
			\hline
			$2^4$  &15.56 & 15.42 \\
			\hline	
			 $2^5$ &16.65 & 16.31  \\
			 \hline	
			 $2^6$ & 19.56&    17.25\\
			 \hline
			 $2^7$ & 22.17 &    22.37\\
			 \hline
			 $2^8$  &22.96 &   22.87 \\	
			\hline
		\end{tabular}%
	\end{center}
	\caption{Effective QTT-rank of the assembled operator $(\tau \textbf{A}+\textbf{I})$ with a given precision $1\times 10^{-12}$.}
	\label{tab:erank3}
\end{table}%

\begin{table}[htbp]
	\begin{center}
		\begin{tabular}{|c|c|c|}
%			\hline
%			 & \multicolumn{2}{c|}{Example 1}& \multicolumn{2}{c|}{Example 2}\\
			\hline
			spatial $N$ on each direction  &  Example 1 & Example 2 \\
%		         spatial $N$ on each direction    &   $(\tau A+ I)$ & $(\tau A+ I)^{\frac{\Delta T}{\tau}}$ & $(\tau A+ I)$ & $(\tau A+ I)^{\frac{\Delta T}{\tau}}$\\
			\hline
			$2^4$  &8.28 & 9.04 \\	
			\hline
			 $2^5$ &9.63 &  12.96 \\	
			 \hline
			 $2^6$ & 12.94&    17.46\\
			 \hline
			 $2^7$ &17.28&   21.88\\
			 \hline
			 $2^8$  &23.47  &   28.17\\	
			\hline
		\end{tabular}%
	\end{center}
	\caption{Effective QTT-rank of the assembled operator $(\tau \textbf{A}+\textbf{I})^{\frac{\Delta T}{\tau}}$ with a given TT-rounding precision $5\times 10^{-4}$ in Example 1 and $5\times 10^{-5}$ in Example 2.}
	\label{tab:erank4}
\end{table}%
From the results in Tables \eqref{tab:erank1}-\eqref{tab:erank2} and \eqref{tab:erank3}-\eqref{tab:erank4}, we find that the QTT-ranks increase very slowly when $N$ increases. Hence, by extracting low-dimensional structures in the solution space, the QTT method helps us alleviate the curse of dimensionality to a certain extent. 

\subsection{Comparison with existing methods}
\noindent
To compute the reference solution, we respectively solve Eqns.\eqref{NonlinearSignalModel_Experiment1} and \eqref{NonlinearSignalModel_Experiment2} using Euler-Maruyama scheme \cite{Platen:1992} with a fine time step, which generates two sequences $X_{t_i}$ and $Y_{t_i}$ of length $dt=0.001$ as discrete real states at time $t_i=idt$, $i=1,...,20000$. We feed the observation $Y_{t_j}$ into the online procedure at each observation time $t_j=j\Delta T$, i.e. only a subsequence $Y_{t_j}$ of $Y_{t_i}$ is regarded as observation sequence and utilized.

To solve the NLF problem in a real time manner, one need to solve the path-wise robust DMZ equation associated with Eqns.\eqref{NonlinearSignalModel_Experiment1} and \eqref{NonlinearSignalModel_Experiment2}. As such, one can solve the FKEs associated with Eqns.\eqref{NonlinearSignalModel_Experiment1} and \eqref{NonlinearSignalModel_Experiment2} using a FD method. However, the FD method becomes expensive when the dimension of the NLF problem increases. We shall show that the QTT method provides considerable savings over the FD method. 

For the FKE associated with the almost linear problem, we restrict the FKE on the domain $[-5,5]^3$ and 
discretize the domain into $2^6$ grids on each dimension. Thus the total degree of freedom is $2^{18}$. The initial distribution is assumed to be a Gaussian function, where the corresponding unnormalized conditional density function is $\sigma_0(\vec{x})=\exp(-4|\vec{x}|^2)$. The time step is chosen to be $\tau=\frac{\Delta T}{100}$ so that the CFL stability condition is satisfied. Based on the initial discretization of $\sigma_0(\vec{x})$, we implement the QTT method to solve the FKE simultaneously, where we fix the TT-rounding precision to be $\epsilon=5\times 10^{-4}$.

Since $\reallywidehat{U}_{\textbf{l},j}^{\frac{\Delta T}{\tau}}$ denotes the predicted solution of the QTT method at different time, which approximates the unnormalized conditional density of the state $\vec{x}(t_j)=[x_1(t_j),x_2(t_j),x_3(t_j)]^T$. Thus, we can compute the estimation of the state in 
three coordinates separately. Specifically, we first compute the Hadamard product of $\reallywidehat{U}_{\textbf{l},j}^{\frac{\Delta T}{\tau}}$ and the QTT-format of each coordinate $x_i$, $i=1,2,3$. 
Then, we compute its total sum and divide it by the total sum of $\reallywidehat{U}_{\textbf{l},j}^{\frac{\Delta T}{\tau}}$. The results corresponding the FD method can be computed similarly. 

In Fig.\ref{fig:CompareTrajectoryLinear}, we show the state estimation results of the almost linear problem in three coordinates separately. The CPU time of the FD method is $2052s$, but the QTT method only requires $15.19s$. A significant computational saving is achieved by our method because most of the operations in the QTT method only logarithmically depend on the total degree of freedom and polynomially depend on the QTT-ranks, which are relatively small; see Tables \eqref{tab:erank1}--\eqref{tab:erank4}. The CPU time of the PF method is $17.36$. The efficiency of the PF method is closely related to the number of particles. In this example, we use 3000 particles to avoid explosion in the tracking, which might happen frequently if only 2000 particles are used. In Fig.\ref{fig:almost_linear_pdf}, we show the profile of the density function at time $t=10$.

%The QTT-method has additional $4.44s$ for offline computing, while the finite difference method and particle filter method do not have offline procedure. In Fig.\ref{fig:almost_linear_pdf}, we show the profile of the density function at time $t=10$.
	
\begin{figure}[tbph] 
	\begin{subfigure}[b]{0.32\textwidth}
		\includegraphics[width=1.1\linewidth]{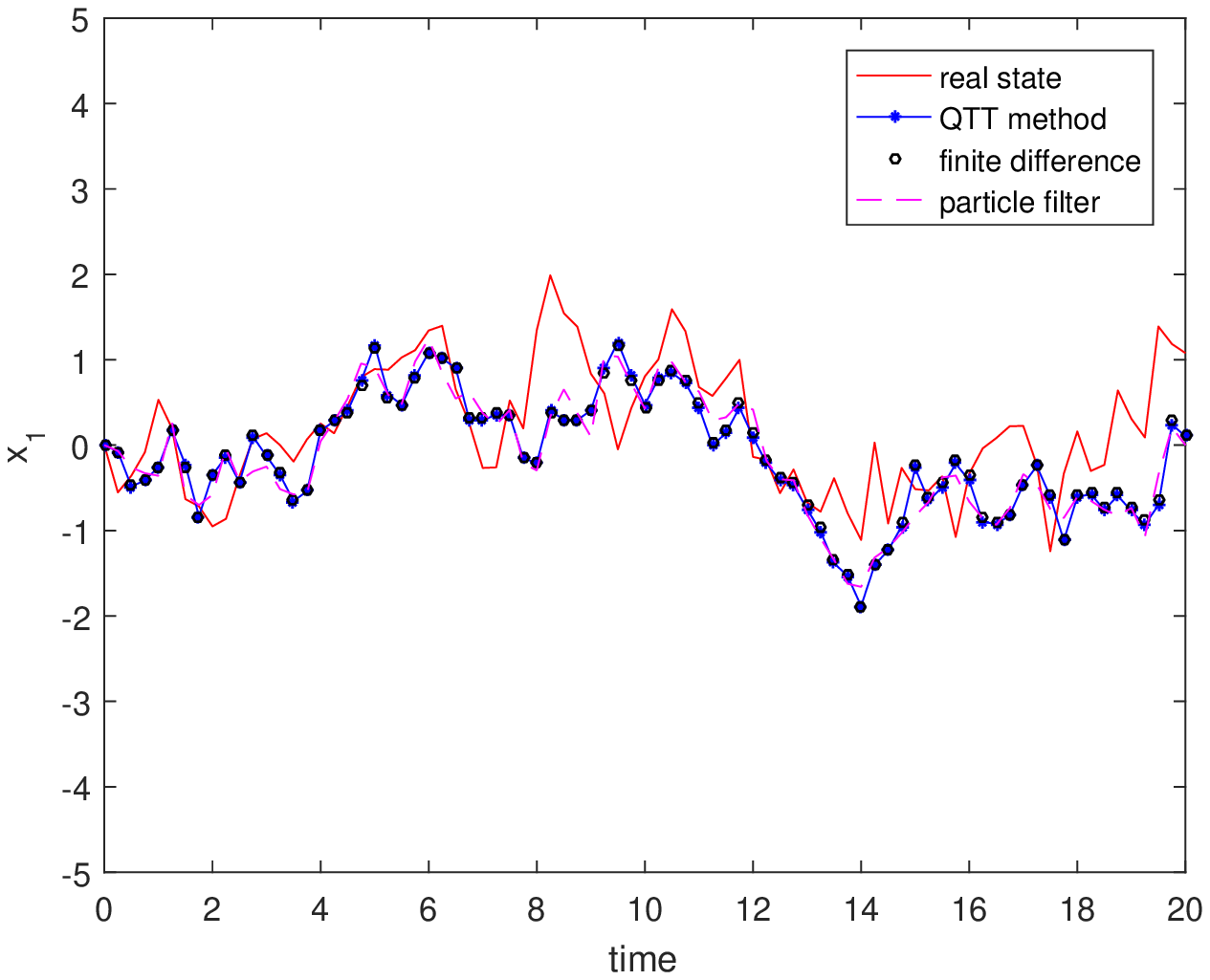} 
		\caption{ $x_1$ component.}
	\end{subfigure}
	\begin{subfigure}[b]{0.32\textwidth}
		\includegraphics[width=1.1\linewidth]{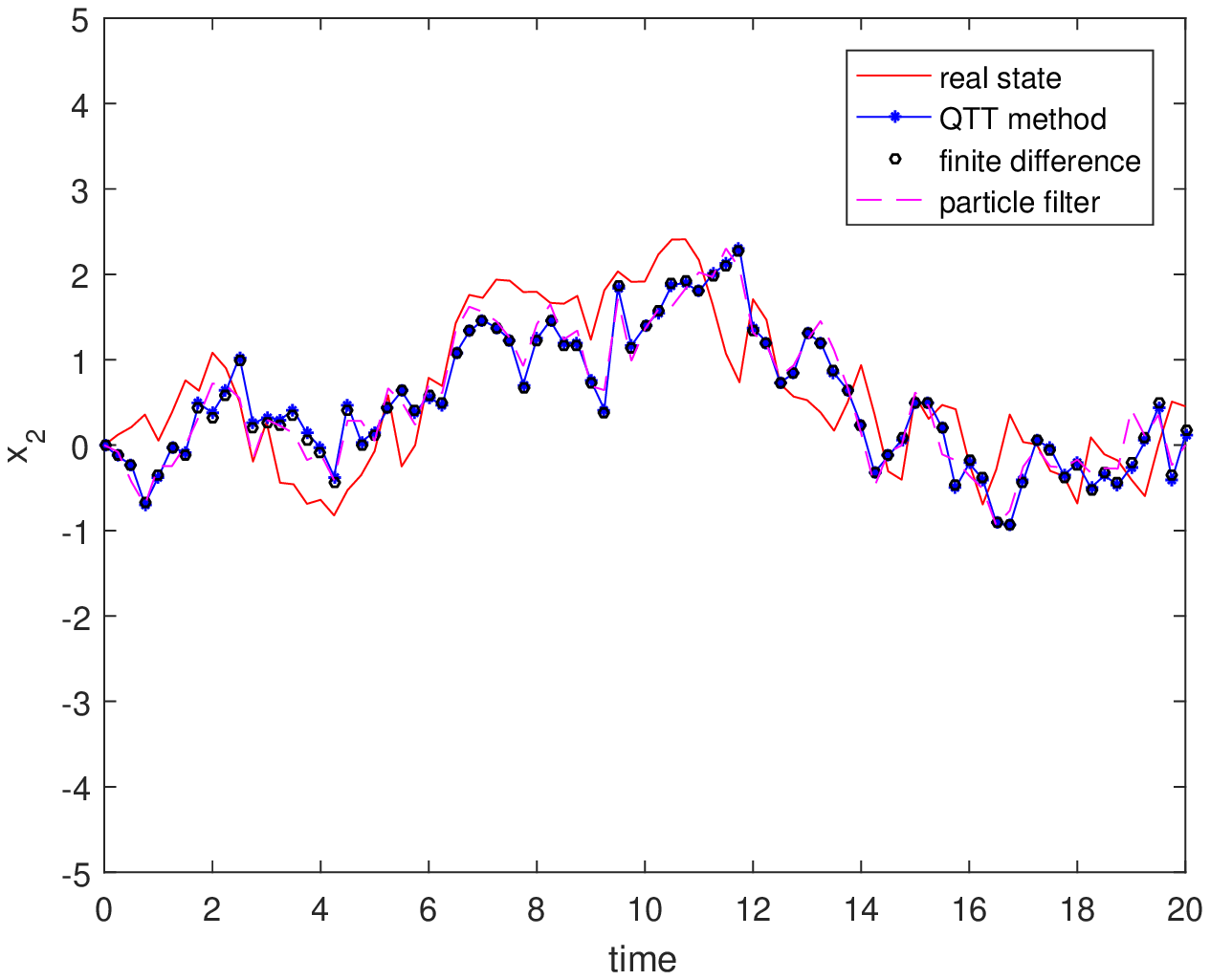} 
		\caption{ $x_2$ component.} 
	\end{subfigure}
	\begin{subfigure}[b]{0.32\textwidth}
		\includegraphics[width=1.1\linewidth]{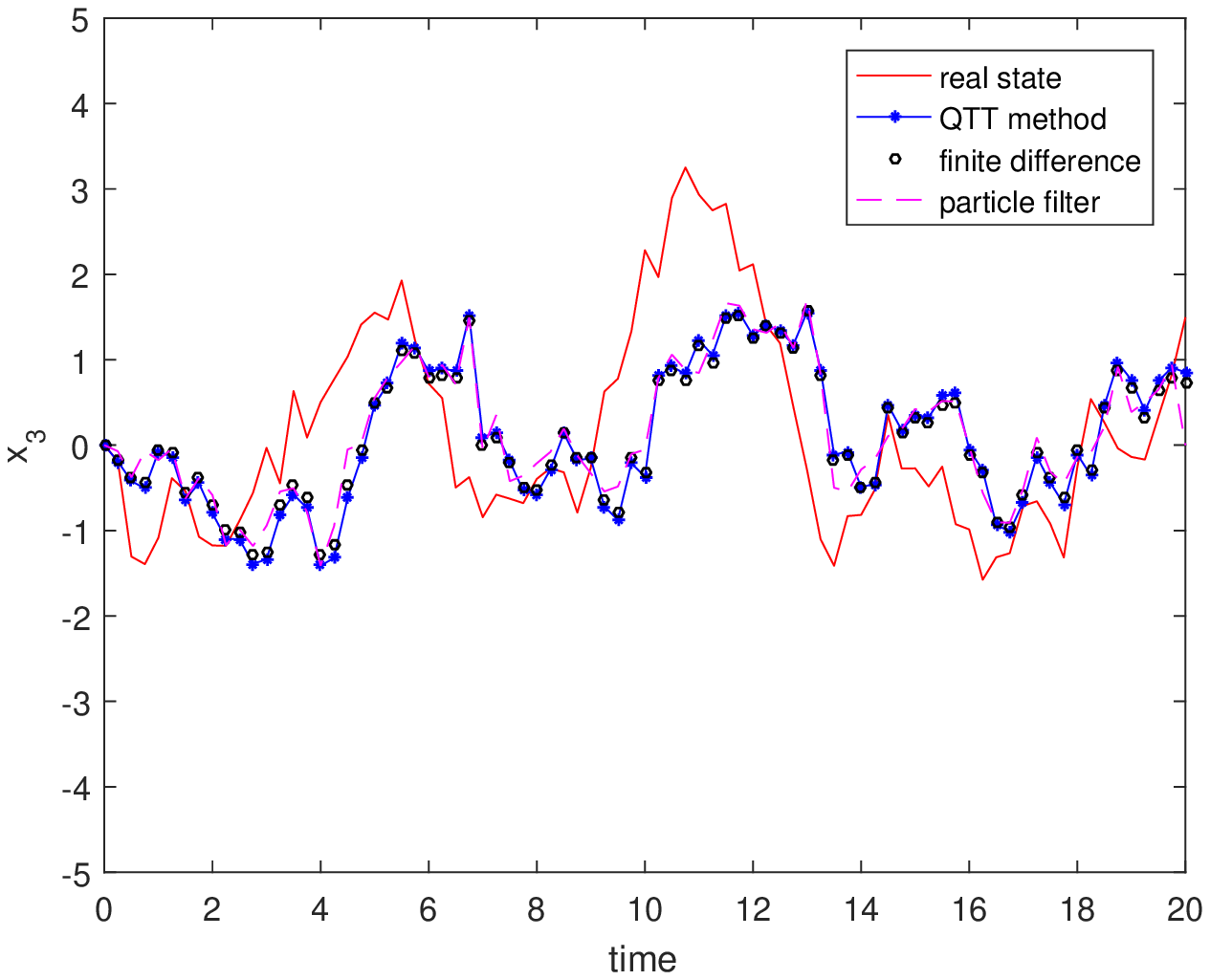} 
		\caption{ $x_3$ component.}
	\end{subfigure}
	\caption{Comparison of a trajectory of the almost linear problem obtained by using different methods.}
	\label{fig:CompareTrajectoryLinear}
\end{figure} 

\begin{figure}[htbp]
\centering
\includegraphics[width=0.55\linewidth]{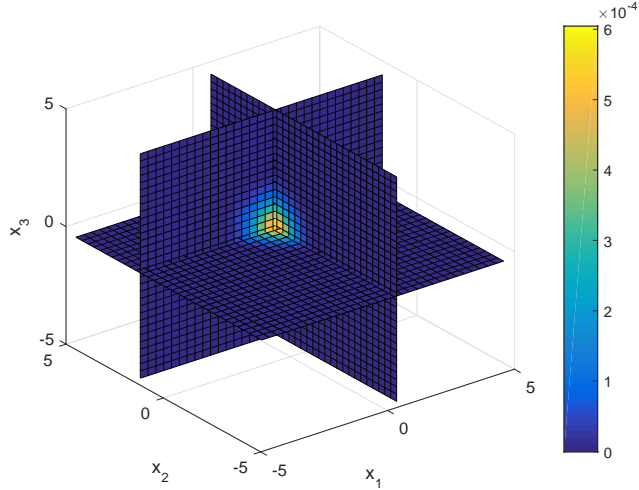}
\caption{The estimations of density function at $t=10s$ in the almost linear problem.}
%fig/tt_almost_linear_pdf_t=10
\label{fig:almost_linear_pdf}
\end{figure} 

For the FKE associated with the cubic sensor problem, we restrict the FKE on the domain $[-3,3]^3$ and discretize the domain into $2^6$ grids on each dimension. Thus, the total degree of freedom is $2^{18}$. The unnormalized density function of the initial state is assumed to be $\sigma_0(\vec{x})=\exp(-10(x_1^4+x_2^4+x_3^4))$. The time step is chosen to be $\tau=\frac{\Delta T}{200}$, which is smaller than the first example, in order to satisfy the CFL stability condition. We fix a higher TT-rounding precision $\epsilon=5\times 10^{-5}$ due to the higher nonlinearity in this example.

In Fig. \ref{fig:CompareTrajectoryCubic}, we show the estimation results of the cubic sensor problem in three coordinates separately. The CPU time of the FD method is $4079s$, while the QTT method is only $17.11s$. The time cost of offline computing of the QTT method is $17.35s$. Even though the QTT-format tensor constructor has linear dependence on the degree of freedom (see Prop.\ref{qtt_online_complexity}), it takes up a minor part of the total computational time. The CPU time of the PF is $29.45s$. In this example, we use 5000 particles to avoid explosion in the tracking. In Fig.\ref{fig:cubic_pdf}, we show the profile of the density function at time $t=12s$.

We repeat the experiment for $N_{path}=100$ times and record the mean square errors (MSEs) averaged over $100$ sample paths. We find that the MSEs between the QTT solution and the FD solution are 0.007 in Example 1 and 0.023 in Example 2, respectively, which show that the FD discretization of the unnormalized density function has a low-rank structure and such structure is approximated well in the QTT-format. Therefore, the QTT method gives a very accurate approximation to the FD solution
with considerable savings.

\begin{figure}[tbph] 
	\begin{subfigure}[b]{0.32\textwidth}
		\includegraphics[width=1.1\linewidth]{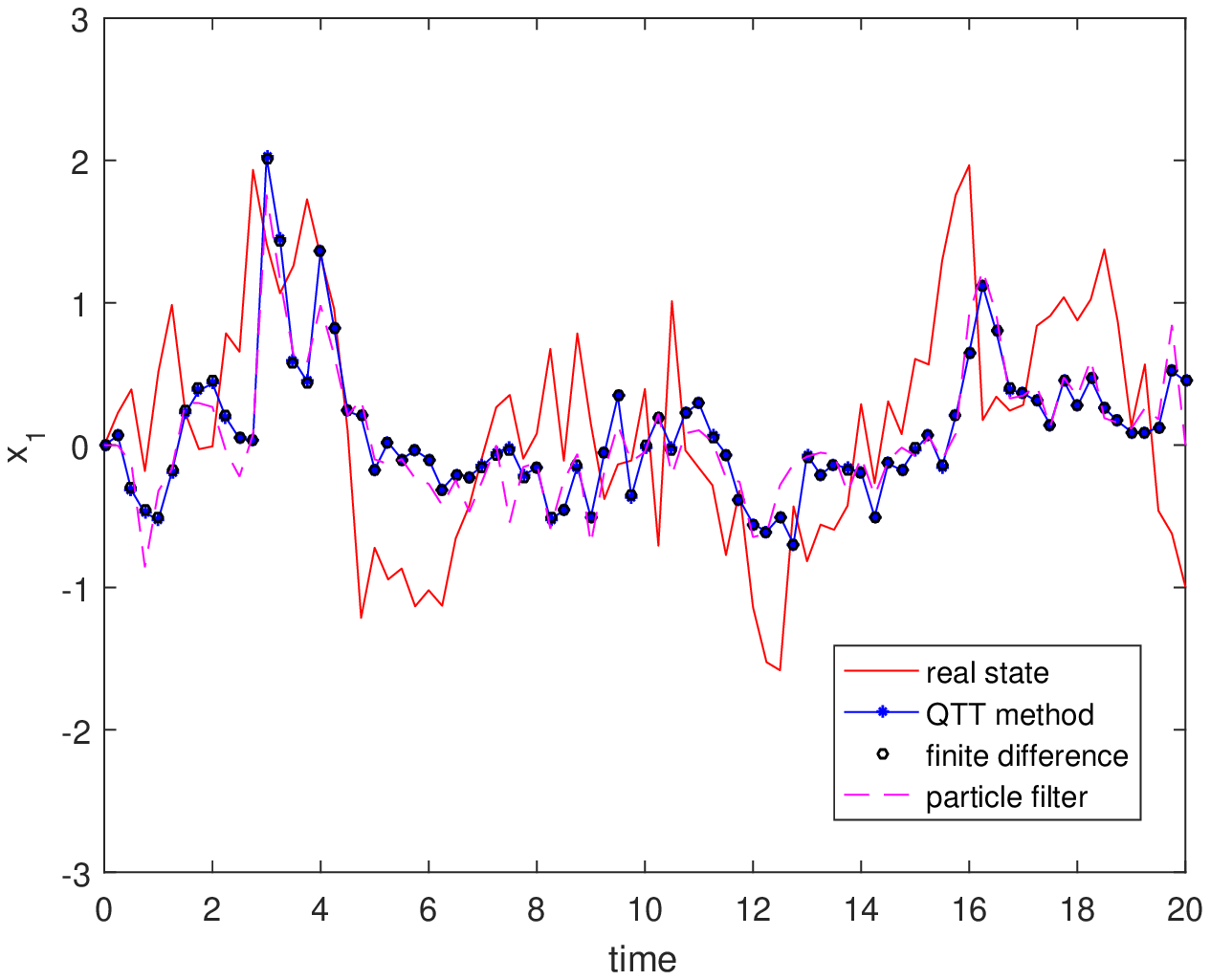} 
		\caption{ $x_1$ component.}
	\end{subfigure}
	\begin{subfigure}[b]{0.32\textwidth}
		\includegraphics[width=1.1\linewidth]{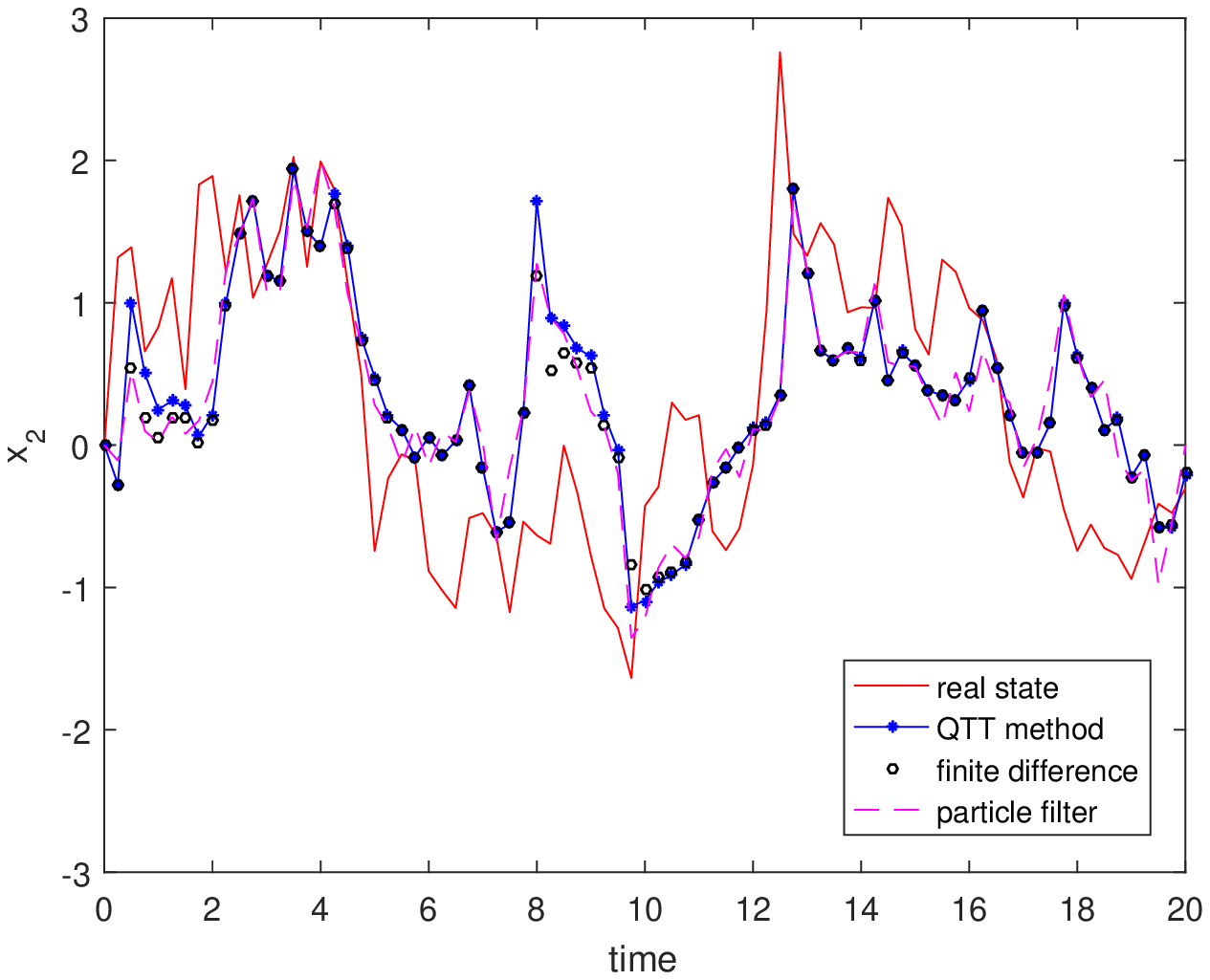} 
		\caption{ $x_2$ component.} 
	\end{subfigure}
	\begin{subfigure}[b]{0.32\textwidth}
		\includegraphics[width=1.1\linewidth]{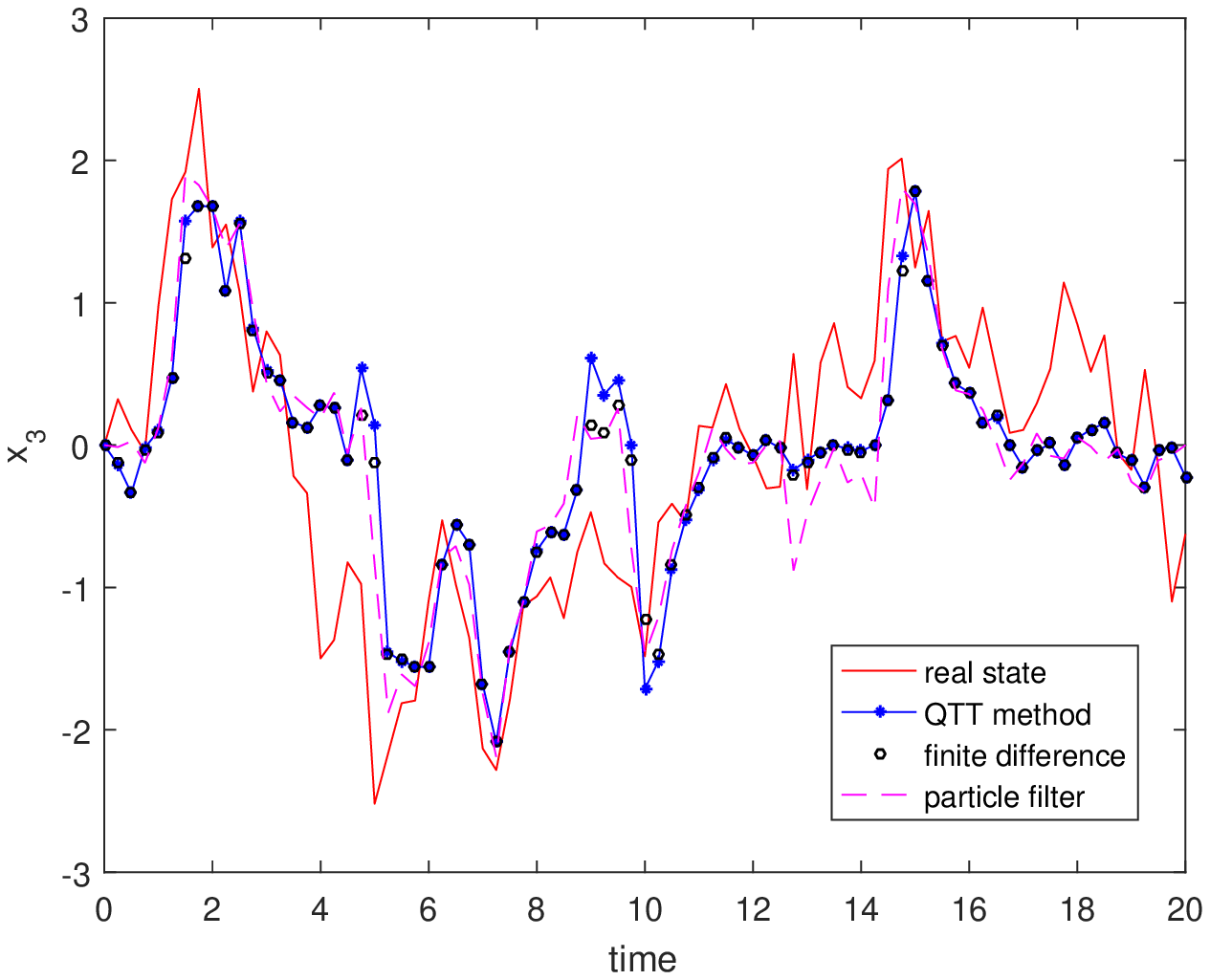} 
		\caption{ $x_3$ component.}
	\end{subfigure}
	\caption{Comparison of a trajectory of the cubic sensor problem obtained using different methods.}
	\label{fig:CompareTrajectoryCubic}
\end{figure} 
\begin{figure}[htbp]
	\centering
	\includegraphics[width=0.55\linewidth]{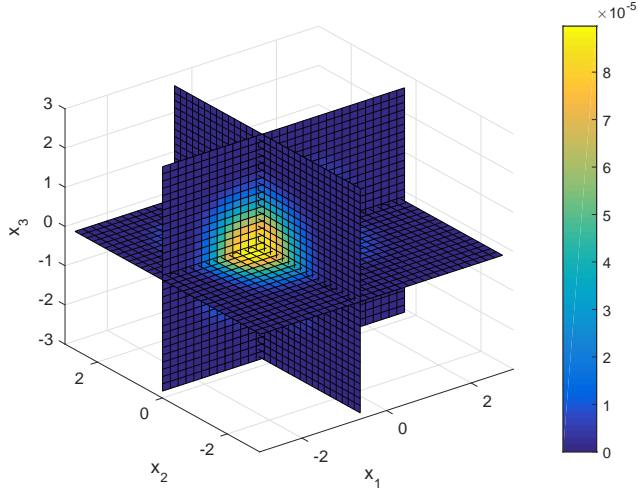}
	\caption{The estimations of density function at $t=12s$ in the cubic sensor problem}
	\label{fig:cubic_pdf}
\end{figure}
\begin{remark}
	The PF method is a very popular method in solving NLF problem, which is a Monte Carlo method and it requires a certain amount of sample to compute statistical quantities. Since the PF method and our method are based on totally different methodologies, we cannot reach a general conclusion about their performances for NLF problems.
\end{remark}
In Fig.\ref{fig:TestConvergence}, we show the convergence of the QTT method with respect to the TT-rounding precision at discrete time points. We use the same settings in the spatial and temporal discretization that were used in the 
experiments before and only change the TT-rounding precision. We compare the estimated states obtained by the QTT method and FD method, which provides the reference solution. One can see that when we decrease the TT-rounding precision, the relative error of the QTT solution decreases accordingly, which agrees with our convergence analysis. In practice, we choose the TT-rounding precision in such a way that we can balance 
the accuracy and computational cost.   
\begin{figure}[tbph] 
	\begin{subfigure}[b]{0.49\textwidth}
		\includegraphics[width=1.1\linewidth]{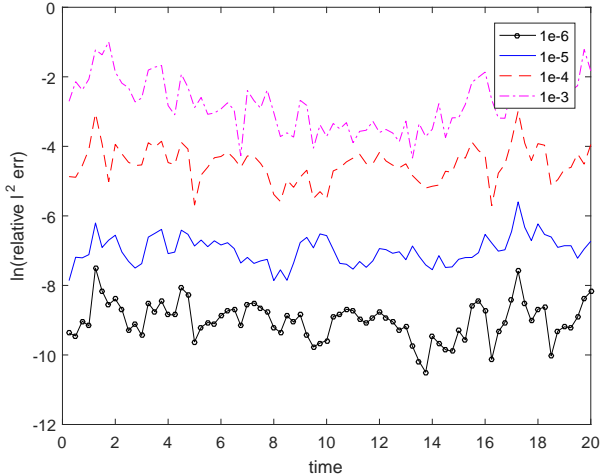} 
		\caption{The almost linear problem.}
	\end{subfigure}
	\begin{subfigure}[b]{0.49\textwidth}
		\includegraphics[width=1.1\linewidth]{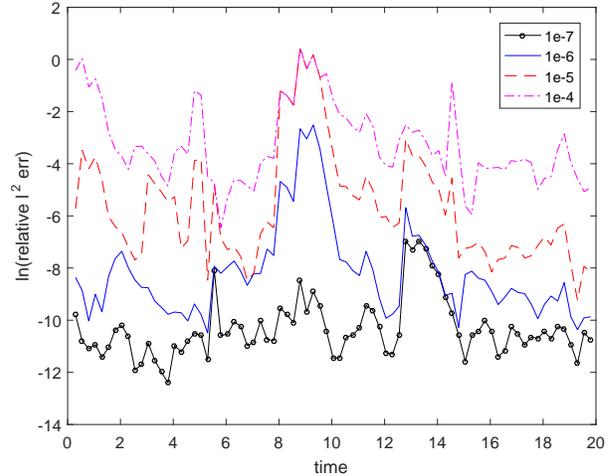} 
		\caption{The cubic sensor problem.} 
	\end{subfigure}
	\caption{Error of estimated states between the QTT method and reference method under different TT-rounding precision.}
	\label{fig:TestConvergence}
\end{figure} 
\subsection{Verification of the computational complexity}\label{sec:verify-cost}
\noindent
In this subsection, we intend to verify the computational complexity studied in the Prop. \ref{qtt_online_complexity}. Notice that the main computational load in the online procedure consists of 
two parts. The first part is $\mathcal{O}(d\log_{2}(N)r^6)$, which comes from the 
QTT operations in solving the FKEs. It polynomially depends on QTT-ranks and logarithmically depends on the degree of freedom. Hence, the QTT method brings in significant savings for high-dimensional problems that have a low-dimensional approximation. The second part of the computational load is $\mathcal{O}(N^dr^2)$, which comes from 
assimilating the observation data into the QTT solution and depends linearly on the degree of freedom 
in the spatial discretization. We find that these two parts are comparable in the 3D numerical experiments that 
were studied in this paper.

Let $t_{FKE}$ denote the CPU time in solving the FKE and $t_{EXP}$ is the CPU time in computing 
the exponential transformation, i.e., assimilating the observation data, respectively. 
In Table \ref{tab:onlinecomplexity}, we show the computational time of the QTT method and FD method in solving the Example 1, respectively. The QTT-rank is the averaged effective QTT-rank of the FKE solution $\widetilde{u}$ at every time.
We find that the growth of computational complexity of the QTT method is significantly slower than that of the FD method. %When QTT operations part is dominant, i.e. $t_{FKE}>t_{EXP}$, the QTT method can achieve almost logarithmic complexity with respect to the degree of freedom.
This numerical experiment shows that the QTT method can solving high-dimensional NLF problems in a real time manner, while the FD method is too expensive. Notice that in Table \ref{tab:onlinecomplexity}, when $N=2^7$ the FD method would cost about 13 hours of computation, which becomes extremely expensive to do this experiment.  

% MSE(mean square error) between two methos is defined as $:=\frac{1}{N_t}\sum_{j=1}^{N_t}||\hat{X}^{QTT}_{t_j} - \hat{X}^{FD}_{t_j} ||^2$, where $N_t=\frac{T}{\Delta T}$,
\begin{table}[htbp]
	\begin{center}
		\begin{tabular}{|c|ccc|cc|}
			\hline
			  & \multicolumn{3}{c|}{QTT method}& \multicolumn{2}{c|}{Finite difference}  \\
			\hline
		        $N$ & $t_{FKE}$  & $t_{EXP}$ & QTT-rank  & $t_{FKE}$  & $t_{EXP}$\\
			\hline
			$2^4$ &   1.87 &0.85& 6.83 &9.62&  0.03\\
			\hline
			 $2^5$   &4.15 &  1.36& 7.88&169 & 0.08\\	
			 \hline
			 $2^6$  &11.18 &   5.20& 8.78 &2802&  0.65\\
			 \hline
			 $2^7$ &   27.49  &46.28 & 8.89 & 47634 & 5.33 \\
			% $2^8$     & 10.48 &  2.96  &  6.76 &289.27 &  0.136\\
			\hline
		\end{tabular}%
	\end{center}
	\caption{CPU time (sec.) of the QTT method and finite difference method. $N$ is the grid number in each dimension. }
	\label{tab:onlinecomplexity}
\end{table}%

%\begin{remark}
%The growth of computational complexity of our QTT method is significantly slower than that of finite difference method. The finite difference method is no longer real time when $N\geq 2^5$ while the QTT method still performs well until $N\geq 2^7$. When QTT operations part is dominant, i.e. $t_{FKE}>t_{EXP}$, our QTT method can achieve almost logarithmic complexity with respect to the degree of freedom. Numerical experiments also demenstrate that such improvement is enough for solving $2$-d or $3$-d nonlinear filtering problems in real time. 
%\end{remark} 

%In practice, the precision of TT-rounding in all of our experiments can be reduced more without much loss of the accuracy. So it is possible to get more efficiency by relaxing the precision. Furthermore, the TT-rounding algorithm also allows for prescribing an upper bound of QTT-ranks of TT-rounding operation, which is useful for assembling the operator matrix in \eqref{AssemblingOperators_FKE}. Note that it is reasonable that QTT-ranks of both $u$ and operator matrix (see Table \ref{tab:erank3}) increase with $N$, so that the growth rate is slightly larger than logarithm.

\section{Conclusions}\label{sec:conclusion}
\noindent
In this paper, we develop an efficient numerical method to solve high-dimensional nonlinear filtering (NLF) problems. Specifically, we use the tensor train decomposition method to solve the forward Kolmogorov equation (FKE) arising from the NLF problem. Our method consists of offline and online stages. In the offline stage, we use the finite difference method to discretize the partial differential operators involved in the FKE and extract low-dimensional structures in the solution space using the tensor train decomposition method. In addition, we approximate the evolution of the FKE operator using the tensor train decomposition method. With the pre-computed and saved low-rank approximation tensors, we achieve fast computing in the online stage given new observation data. Under some mild assumptions, we provide convergence analysis for the proposed method. Our analysis result reveals different sources of errors and provides some guidance on the implementation of our method so that the error is controllable. Finally, we present numerical results to verify the efficiency and accuracy of the proposed method in solving 3D NLF problems. Numerical results show that the solutions of the FKEs indeed have certain low-dimensional structures. By using the tensor train decomposition method to extract the low-dimensional structures in the solution space of the FKE, we succeed in solving the high-dimensional NLF problems in a real-time manner.

There are two directions we want to explore in our future work. First, we are interested in developing  efficient numerical methods for high-dimensional NLF problems (with $d>3$), which will be reported in our subsequent work. In addition, we will develop numerical methods to solve NLF problems, where the drift and observation functions are time-dependent. This type of problem is more difficult since the potential low-dimensional structures in the solution space may vary with respect to time. One may need to develop some dynamically low-dimensional approximation methods to 
address this issue; see e.g. \cite{ChengHouZhang1:13,ChengHouZhang2:13}.

\section{Acknowledgements}
\noindent
The research of S. Li is partially supported by the Doris Chen Postgraduate Scholarship. The research of Z. Wang is partially supported by the Hong Kong PhD Fellowship Scheme. 
%The work of X. Luo was supported by the National Natural Science Foundation of China (11501023). 
The research of S. S.-T. Yau was supported by the National Natural Science Foundation of China (11471184), Tsinghua University Education Foundation fund (042202008), and a start-up fund from Tsinghua University.  The research of Z. Zhang is supported by the Hong Kong RGC grants (Projects 27300616, 17300817, and 17300318), National Natural Science Foundation of China (Project 11601457), Seed Funding Programme for Basic Research (HKU), and Basic Research Programme (JCYJ20180307151603959) of The Science, Technology and Innovation Commission of Shenzhen Municipality.

\appendix
%\section{Adaptive algorithm for active basis index $\sJ$} \label{sec:appendix_Adaptive_gPC_index}

%\section*{References}
%\bibliographystyle{plain}
\bibliographystyle{siam}
%\bibliographystyle{amsalpha}
%\bibliographystyle{ieeetr}
%\bibliography{ZWpaper}

\end{document}